\documentclass[11pt,dvipsnames,reqno]{amsart}
\usepackage[left=3.3cm,right=3.3cm,bottom=3.5cm,top=3.4cm]{geometry}

\usepackage{style}
\usepackage{times}
\usepackage[all]{xy}
\usepackage{pdflscape}
\usepackage{colortbl}

\newcommand{\calL}{\mathcal{L}}
\newcommand{\calO}{\mathcal{O}}
\newcommand{\calP}{\mathcal{P}}

\newcommand{\calC}{\mathcal{C}}
\newcommand{\calH}{\mathcal{H}}
\newcommand{\bbP}{\mathbb{P}}
\newcommand{\bbZ}{\mathbb{Z}}

\newcommand{\sfE}{\mathsf{E}}

\newcommand{\bbF}{\mathbb{F}}

\newcommand{\half}{\frac{1}{2}}

\newcommand{\frakS}{\mathfrak{S}}
\newcommand{\frakA}{\mathfrak{A}}

\newcommand{\SU}{\operatorname{SU}}
\newcommand{\U}{\operatorname{U}}
\newcommand{\Or}{\mathrm{O}}
\newcommand{\PSU}{\operatorname{PSU}}

\newcommand{\SO}{\operatorname{SO}}
\newcommand{\PSp}{\operatorname{Sp}}
\newcommand{\PSL}{\operatorname{PSL}}
\newcommand{\PGL}{\operatorname{PGL}}

\newcommand{\la}{\langle}
\newcommand{\ra}{\rangle}

\newcommand{\fppf}{\textrm{fppf}}
\newcommand{\et}{\textrm{\'et}}
\newcommand{\red}{\rm{red}}

\newcommand{\beq}{\begin{equation}}
\newcommand{\eeq}{\end{equation}}

\renewcommand{\arraystretch}{1.15}

\theoremstyle{plain}
\newtheorem*{theorem*}{Theorem}

\author{Igor Dolgachev}
\address{\hfill \newline 
Department of Mathematics \newline
University of Michigan \newline
525 East University Avenue \newline
Ann Arbor,
 MI 48109-1109 USA}
\email{idolga@umich.edu}

\author{Gebhard Martin}
\address{\hfill \newline
Mathematisches Institut  \newline
Universit\"at Bonn \newline
Endenicher Allee 60 \newline
53115 Bonn \newline
Germany}
\email{gmartin@math.uni-bonn.de}

\begin{document}
\title[Automorphisms of del Pezzo surfaces in characteristic 2]{Automorphisms of del Pezzo surfaces in characteristic 2}

\date{}
\begin{abstract} We classify the automorphism groups of del Pezzo surfaces of 
degrees one and two over an algebraically closed field of characteristic two. 
This finishes the classification of automorphism groups of del Pezzo surfaces 
in all characteristics.
\end{abstract}

\maketitle

 \tableofcontents
\section*{Introduction} 
This is a continuation of our paper \cite{DM}, where we 
finished the classification of the automorphism groups of del Pezzo surfaces over an algebraically closed field 
of positive characteristic $p\ne 2$. In this paper we treat the remaining case when the characteristic is equal to $2$.

As we explained in the introduction to \cite{DM}, the remaining part of the classification concerns del 
Pezzo surfaces of degree one and two. The cases of odd and even positive characteristic are drastically different since 
in the latter case the anti-canonical map (resp. the anti-bicanonical map) is a separable Artin-Schreier cover  
of degree two 
but not
a Kummer cover as in the cases of odd characteristic. So, no plane quartic curves (and no canonical genus $4$ curves with 
vanishing theta characteristic) appear as branch curves.

Instead, in characteristic $2$, the branch curve $B$ of the anti-canonical (resp. anti-bicanonical) map is a not necessarily smooth plane conic (resp. a cubic in $\mathbb{P}^3$). The ramification curve $R$ is a purely inseparable cover of $B$. In Theorems \ref{canform} and \ref{thm: canform}, we give normal forms for del Pezzo surfaces of degree $2$ and $1$ depending on the singularities of $R$ and $B$.

Although plane quartics and canonical curves of genus four disappear in characteristic $2$, their 
familiar attributes like $28$ bitangent lines or $120$ tritangent planes persist. We call them \emph{fake bitagents} and \emph{fake tritangent planes}. 
They are defined to be lines in the plane (resp. planes in the three-dimensional space) which split under the anti-canonical (resp. anti-bicanonical) map.

It is well-known that the blow-up of the anti-canonical base point on a del Pezzo surface of degree $1$ yields a rational elliptic surface with only irreducible fibers and, conversely, the contraction of a section of a rational elliptic surface with only irreducible fibers yields a del Pezzo surface of degree $1$. Thus, the normal forms of Theorem \ref{thm: canform} also give normal forms for all rational elliptic surfaces with only irreducible fibers.

Quite surprisingly, in characteristic $2$, also every del Pezzo surface of degree $2$ has a canonically associated rational elliptic surface. This surface is obtained by blowing up the base points of the preimage of the pencil of lines through the \emph{strange point} of the branch locus $B$. We study the properties of this \emph{strange fibration} in Section \ref{sec: strangefibration}.

Using these geometric observations, we classify the automorphism groups of all del Pezzo surfaces of degree $2$ and $1$ in characteristic $2$.
The following result is proved in Theorems \ref{thm: main2} and \ref{thm: main1}.

 \begin{theorem*}
A finite group $G$ is realized as the automorphism group $\Aut(X)$ of a del Pezzo surface $X$ of degree $2$ (resp. $1$) over an algebraically closed field $k$ of characteristic ${\rm char}(k)= 2$ if and only if $G$ is listed in Table \ref{tbl:autodp2} (resp. Table \ref{tbl:autodp1}) in the Appendix.
 \end{theorem*}

Table \ref{tbl:autodp2} (resp. Table \ref{tbl:autodp1}) also gives the conjugacy classes in $W(E_7)$ (resp. $W(E_8)$) of all elements of $\Aut(X)$ for all del Pezzo surfaces $X$ of degree $2$ (resp. degree $1$). 
We refer to \cite{DM} for a general discussion of the history of 
the problem and its relationship to the classification of 
conjugacy classes of finite subgroups of the 
planar Cremona group. Also the reader finds there some general facts about del Pezzo surfaces, e.g. the relationship with the Weyl 
groups of roots systems and some classification results from the group theory.

\newpage

\section{Notation}

For the convenience of the reader, we recall the notations for some finite groups which we will 
encounter in this article. Throughout this article, $p$ is a prime and $q$ is a power of $p$. 
Unless stated otherwise, $\Bbbk$ denotes an algebraically closed field of characteristic $2$.

\begin{itemize}
\item $C_n$ is the cyclic group of order $n$.
\item $\frakS_n$ and $\frakA_n$ are the symmetric and alternating groups on $n$ letters.
\item $Q_8$ is the quaternion group of order 8.
\item $D_{2n}$ is the dihedral group of order $2n$.
\item $n^k = (\bbZ/n\bbZ)^k$. In particular, $n = n^1 = \bbZ/n\bbZ$.
\item $p_{\pm}^{1+2n}$, the extra special group. For odd $p$ the sign $+$ ($-$) defines a group of exponent 
$p$ ($p^2$). For $p = 2$, the sign distinguishes the type of the quadratic forms on $2^{2n} = \bbF_2^{2n}$ 
defined by the extension.
\item $\GL_n(q) = \GL(n,\bbF_q)$.
\item $\PGL_n(q) = \GL_n(q)/\bbF_q^*$. Its order is $N = q^{\half n(n-1)}(q^n-1) \cdots (q^2-1)$. 
\item $\SL_n(q) = \{g\in \GL_n(q):\det(g) = 1\}$. This is a subgroup of $\GL_n(q)$ of index $(q-1)$.
\item ${\rm L}_n(q) =\PSL_n(q)$ is the image of $\SL_n(q)$ in $\PGL_n(q)$. Its order is $N/(q-1,n)$. 
\item For odd $n$, $\Or_{n}(q)$ is the subgroup of $\GL_n(q)$ that preserves a non-degenerate quadratic form $F$.
\item For even $n$, $\Or_n^+(q)$ (resp. $\Or_n^-(q)$) is the subgroup of $\GL_n(q)$ that preserves a non-degenerate quadratic form $F$ of Witt defect $0$ (resp. $1$).
\item $\SO_{n}^{\pm}(q)$ is the subgroup of $\Or_{n}^{\pm}(q)$ of elements with determinant $1$. 
\item ${\rm PSO}_n^{\pm}(q)$ is the quotient of $\SO_{n}^{\pm}(q)$ by its center.
\item $\Sp_{2n}(q)$ is the subgroup of $\SL_q(2n)$ preserving the standard symplectic form on $\bbF_q^{2n}$.
Its order 
is $q^{n^2}(q^{2n-1}-1)\cdots (q^2-1)$.
\item $\PSp_{2n}(q) = \Sp_{2n}(q)/(\pm 1)$.
\item $\SU_n(q^2)$ is the subgroup of $\SL_n(q^2)$ of matrices preserving the hermitian form 
$\sum_{i=1}^nx_i^{q+1}$. 
Its order is 
$q^{\half n(n-1)}(q^n-(-1)^n)(q^{n-1}-(-1)^{n-1})\cdots (q^3+1)(q^2-1)$.
 We have $\SU_2(q^2) = \SL_2(q).$
\item $\PSU_n(q^2) = \SU_n(q^2)/C, $ where $C$ is a cyclic group of order $(q+1,n)$ of diagonal Hermitian matrices.
The simple group $\PSU_n(q^2)$ is denoted by 
$\U_n(q)$ in \cite{Atlas}.
\item $\mathcal{H}_3(3)$ is the Heisenberg group of $3\times 3$ upper triangular matrices with entries in $\bbF_3$.
\item $A.B$ is a group that contains a normal subgroup $A$ with quotient group $B$.
\item $A:B$ is the semi-direct product $A\rtimes B$.
 \end{itemize} 

\newpage

\section{Del Pezzo surfaces of degree $\ge 3$}
For the convenience of the reader, we first recall the classification of automorphism groups of del Pezzo surfaces of degree at least $3$.

\subsection{Degree $\ge 5$} 
For del Pezzo surfaces of degree at least $5$, the description of $\Aut(X)$ is characteristic-free. We refer the reader to \cite[Section 3]{DM} or \cite{CAG} for details.

\subsection{Quartic del Pezzo surfaces}
Starting from degree $4$, the classification of automorphism groups depends on the characteristic. 
As in the other characteristics, a quartic del Pezzo surface $X$ is a blow-up of $5$ points in $\mathbb{P}^2$ no $3$ of which are collinear. Moreover, the anti-canonical linear system 
$|-K_X| = |\calO_{\bbP^2}(3)-p_1-p_2-p_3-p_4-p_5|$ embeds $X$ into $\bbP^4$ as a complete intersection of 
two quadrics.

Since $p = 2$, these quadrics cannot be diagonalized. Instead, as shown in \cite{DD}, one can choose the normal forms
\begin{equation} \label{quartcdP} (ab+b+1)t_2^2+at_3^2+t_2t_3+t_3t_4 =
bt_1^2+(ab+a+1)t_2^2+t_1t_3+t_2t_4 = 0,
\end{equation}
where $a,b$ are parameters such that the binary form $\Delta = uv(u+v)(u+av)(bu+v)$ has five distinct roots.
 
As in the case $p \neq 2$, the automorphism group $\Aut(X)$ contains a normal subgroup $H$ isomorphic to $2^4$, and the quotient $G = \Aut(X)/H$ is isomorphic to a subgroup of $\frakS_5$. 
The classification is summarized in Table \ref{tbl:autodp4} in the Appendix. There, the first column refers to the values of the parameters $a$ and $b$ in Equation \eqref{quartcdP} above. The conjugacy classes of elements of $\Aut(X)$ can be obtained by combining \cite[Table 2]{DD} and \cite[Table 5]{Carter}. 

\subsection{Cubic surfaces} \label{S:3.3}
The classification of automorphism groups of cubic surfaces in characteristic $2$ was achieved in \cite[Table 7]{DD}. For the convenience of the reader, we recall the classification in Table \ref{tbl:cubics} in the Appendix.

\section{Del Pezzo surfaces of degree $2$} \label{sec: deg2}
\subsection{The anti-canonical map} \label{SS: 2.1}
We start by describing the geometry of del Pezzo surfaces of degree $d = 2$ over an algebraically closed field $\Bbbk$ of 
characteristic $p = 2$. We refer to \cite{Demazure} for the basic facts from the theory of 
del Pezzo surfaces over fields of any characteristic. It is known that the anti-canonical 
linear system $|-K_X|$ has no base points and defines a finite morphism $f:X\to \bbP^2$ of degree 2.

If $p\ne 2$, the map $f$ is 
automatically separable and its branch curve is a smooth plane quartic. So any automorphism of $X$ induces an 
automorphism of the quartic, and, conversely, any automorphism of the quartic can be lifted to two automorphisms of $X$
that differ by the deck transformation, classically called the Geiser involution. 

If $p = 2$, the structure of $f$, being a morphism of degree $2$, is more complicated. 
Nevertheless, as a first step, we observe that $f$ is still always separable.

\begin{proposition} \label{prop: separable}
The anti-canonical linear system $|-K_X|$ defines a finite separable morphism $f: X \to \mathbb{P}^2$ of degree $2$.
\end{proposition}

\begin{proof}
Assume that $f$ is not separable. Then, since ${\rm deg}(f) = 2$, $f$ is purely inseparable. Hence, 
$f$ is a homeomorphism in the \'etale topology, which is absurd since 
$H_{\et}^2(X,\mathbb{Z}_{\ell})$ has rank $8$ (because $X$ is the blow-up of 
7 points in the plane), 
while $H_{\et}^2(\mathbb{P}^2,\mathbb{Z}_{\ell})$ has rank $1$.
\end{proof}

Let 
$$R(X,-K_X) = \bigoplus_{n=0}^\infty H^0(X,\calO_X(-nK_X))$$
be the graded anti-canonical ring of $X$. By the Riemann-Roch Theorem, $\dim_\Bbbk R(X,-K_X)_1 = 3$ and 
$\dim_\Bbbk R(X,-K_X)_2 = 7$.  One can show that $R(X,-K_X)$ is generated by  $R(X,-K_X)_1$ and one element
from $R(X,-K_X)_2$ that does not belong to the symmetric square of $R(X,-K_X)_1$. Let $x,y,z$ be elements of 
$R(X,-K_X)_1$ and $w\in R(X,-K_X)_2$ which together generate $R(X,-K_X)$. Then, the relation between the generators 
is of the form
\beq\label{equation}
w^2 + A(x,y,z)w + B(x,y,z) = 0,
\eeq
where $A$ and $B$ are homogeneous forms of degree $2$ and $4$, respectively. 
In particular, via equation \eqref{equation}, we can view $X$ as a surface of degree $4$ in 
the weighted projective space $\mathbb{P}(1,1,1,2)$, and the anti-canonical map is the projection of this surface onto the $x,y,z$-coordinates.

If $p\ne 2$, we can complete the square, get rid of $A$, and obtain the standard equation of a del Pezzo surface of degree
$2$. The curve $V(B(x,y,z))$ is the smooth plane quartic we mentioned in the introduction. 
The Geiser involution just negates $w$. 

In our case, when $p = 2$, we cannot get rid of $A$, for otherwise the map would become inseparable. Also, the coefficient $B$ is not uniquely determined, since replacing $w$ with $w+Q$ for any quadratic form $Q$ changes $B$ to $B+AQ+Q^2$, without changing the isomorphism class of the surface. Taking $Q= A$, we obtain the analog of the 
 Geiser involution, so we keep the name for this involution.

The non-uniqueness of $B$ becomes more natural if we take the following different point of view:
By \cite[Proposition 1.11]{Ekedahl}, the double cover $f$ is a torsor under a group scheme $\alpha_{\calL,s}$ of order 2 over $\bbP^2$, defined by the exact sequence of fppf-sheaves 
 $$0\to \alpha_{\calL,s} \to \calL \overset{\phi}{\to} \calL^{\otimes 2} \to 0$$
for some line bundle $\calL$ and a global section $s$. The homomorphism of sheaves 
$\phi$ is locally given by $a\mapsto a_U^2+a_Us_U$, so $s$ cuts out the branch locus of $f$. By \cite[Proposition 1.7]{Ekedahl}, 
we have $\omega_X \cong f^*(\mathcal{O}_{\bbP^2}(-3) \otimes \calL^{-1})$, 
hence $\calL \cong \mathcal{O}_{\bbP^2}(2)$ and $s = A$.  The $\alpha_{\calL,s}$-torsor
corresponding to $f$ is defined by a cohomology class 
in $H^1_{\fppf}(\bbP^2,\alpha_{\calL,s})$. Since $H^1_{\fppf}(\bbP^2,\calL) = H^1(\bbP^2,\calL) = 0$, we have 
$$H_{\fppf}^1(\bbP^2,\alpha_{\calL,s}) \cong H^0(\bbP^2,\calL^{\otimes 2})/\wp(H^0(\bbP^2,\calL)),
$$
where $\wp = H^0(\phi)$. The ternary form $B$ is a representative of this space, and hence it 
is defined only up to a transformation of the form $B\mapsto B+Q^2+AQ$, where $Q$ is a quadratic form in $x,y,z$.

By writing the equation of $X$ locally as 
$w_U^2+a_Uw_U+b_U$, and taking partial derivatives, we see that the differentials
$w_Uda_U+db_U$ restricted to $V(A)$ glue together to define a global section 
$\alpha$ of $\Omega_{\bbP^2}^1\otimes \calL^{\otimes 2}\otimes \calO_{V(A)}$. This
section vanishes  if and only if $X$ is singular. 
So, in our case, when $X$ is assumed to be smooth, we obtain the following.

\begin{proposition}\label{smooth} In \eqref{equation}, the equations
$$A = 0,\quad wA_x+B_x = 0,\quad wA_y+B_y = 0, \quad wA_z+B_z = 0$$
have no common solutions.

\end{proposition}


\subsection{Normal forms}
Recall that $X$ is given by an equation of the form
$$
w^2 + Aw + B = 0,
$$
where $A$ is a quadratic ternary form and $B$ is quartic ternary form. 
We say that $V(A)$ is the branch curve of the cover, 
and its preimage $R = f^{-1}(V(A))$ under the anti-canonical map $f:X\to \bbP^2$ 
will be called the \emph{ramification curve}. 

\begin{remark}
We use the notation $A_{2n}$ for singularities of curves whose formal completion is 
isomorphic to the unibranched singularity $y^2+x^{2n+1} = 0$. If $n = 1$, this is an ordinary cusp singularity. These are exactly the curve singularities that can occur on reduced purely inseparable double covers of smooth curves in characteristic $2$. Indeed, complete locally, such a double cover is given by an equation of the form
$
y^2 + ux^m,
$
where $u \in k[[x]]$ is a unit. Now, we can apply a substitution of the form $y \mapsto y + f$ for a suitable power series $f$ to assume that $m$ is odd and then replace $x$ by $\lambda x$, where $\lambda$ is an $m$-th root of $u^{-1}$, which exists by Hensel's lemma. In other words, the singularity defined by $y^2 + ux^m$ is of type $A_{2n}$, where $2n+1$ is the smallest odd power of $x$ that occurs in $ux^m$.
\end{remark}

The following theorem gives normal forms for the cover $f: X \to \bbP^2$. 
In total, we obtain six normal forms, corresponding to the six possible combinations of singularities of $V(A)$ and $R$.

\begin{theorem}\label{canform} Every del Pezzo surface of degree $2$ in characteristic $2$ is a quartic surface in $\mathbb{P}(1,1,1,2)$ given by an equation of the form
$$
w^2 + A(x,y,z)w + B(x,y,z),
$$
where $(A,B)$ is one of the following:
 $$
  \resizebox{\textwidth}{!}{
  $
    \begin{array}{|c|c|c|c|c|c|} \hline
    \text{Name}  & A & B & B_1 & B_0 & \# \text{Parameters}  \\ \hline \hline
    (1) (a)  & x^2+yz & xB_1 + B_0 & \lambda yz(y+z) & ay^4 + by^3z + cy^2z^2 + dyz^3 + ez^4 & 6\\ \hline
    (1) (b) & x^2+yz & xB_1 + B_0 & y^2z & ay^4 + by^3z + cy^2z^2 + dyz^3 + ez^4 & 5\\ \hline
   (1) (c) & x^2 + yz & xB_1 + B_0 & y^3 &  by^3z + cy^2z^2 + dyz^3 + ez^4 & 4 \\ \hline
    (2) (a) & xy & B_1 + B_0^2 & xz^3 + yz^3 & ax^2 + by^2 + cz^2 + dxz + eyz & 5  \\ \hline
    (2) (b) & xy & B_1 + B_0^2 & xz^3 + y^3z & ax^2 + cz^2 + dxz + eyz & 4 \\ \hline
    (3) \phantom{ (c) } & x^2 & xB_1 + B_0 & z^3 + ayz^2 & y^3z + by^2z^2 + cz^4 & 3 \\ \hline
    \end{array}
    $
    }
    $$
The parameters satisfy the following conditions:
\begin{enumerate}
    \item[(1) (a)] $\lambda \neq 0, \lambda^2 + a + b + c + d + e \neq 0, b^2 + a \neq 0, d^2 + e \neq 0$.
    \item[(1) (b)] $b^2 + a \neq 0, d^2 + e \neq 0$.
    \item[(1) (c)] $d^2 + e \neq 0$.
    \item[(2) (a)] $a \neq 0,b \neq 0$.
    \item[(2) (b)] $a \neq 0$.
    \item[(3) \phantom{( ) }] None.
\end{enumerate}
In terms of these normal forms, the singularities of the irreducible components of $R_{\red}$ are as follows:
\begin{enumerate}
    \item[(1) (a)] Three $A_2$-singularities, over $[0:1:0]$, $[0:0:1]$ and $[1:1:1]$.
    \item[(1) (b)] An $A_4$-singularity over $[0:0:1]$ and an $A_2$-singularity over $[0:1:0]$.
    \item[(1) (c)] An $A_6$-singularity over $[0:0:1]$.
    \item[(2) (a)] Two $A_2$-singularities, over $[1:0:0]$ and $[0:1:0]$.
    \item[(2) (b)] Two $A_2$-singularities, over $[1:0:0]$ and $[0:0:1]$.
    \item[(3) \phantom{( ) }] An $A_2$-singularity over $[0:0:1]$.
\end{enumerate}

\end{theorem}

\begin{proof} Since $f: X \to \bbP^2$ is separable, $A$ is non-zero. Hence, up to projective equivalence, there are three possibilities for $A$, corresponding to $(1)$, $(2)$, and $(3)$. Now, we study those cases separately. The conditions on the parameters will follow from Proposition \ref{smooth} by computing partial derivatives, a task which we will leave to the reader.

\smallskip
(1) $A = x^2 + yz$
\smallskip

Applying a substitution of the form $w \mapsto w + Q$ for a suitable quadratic form $Q$ allows us to assume that $B = xB_1 + B_0$ for homogeneous forms $B_0$ and $B_1$ in $y$ and $z$.

Let $x= uv, y = u^2, z= v^2$ define the Veronese 
isomorphism between $V(A)$ and $\bbP^1$. Substituting in $B$, we get that $R$ is isomorphic to the double cover 
of $\bbP^1$ given by the equation 
$$w^2+ uvB_1(u^2,v^2) + B_0(u^2,v^2) = 0.$$
By taking the partials, we find that $R$ is singular exactly over the roots of $B_1$.

After applying a suitable substitution that preserves $A$, we can move these roots to special positions. Note that the substitution $w \mapsto w + Q$ of the first paragraph does not change the position of these singularities, so we can still assume that $B = xB_1 + B_0$.

If the roots are distinct, we get Case (a), if there are two distinct roots, we get Case (b), and if there is only a single root, we get Case (c). Note that in Cases (b) and (c), the substitution $y \mapsto \lambda y, z \mapsto \lambda^{-1}z$ preserves the location of the roots and scales $B_1$, which is why we can assume that $xB_1$ occurs with coefficient $1$. Finally, in Case (c), we can apply a substitution of the form $z \mapsto z + \lambda^2 y, x \mapsto x + \lambda y$ for a suitable $\lambda$ to assume that $B_0(1,0) = 0$.

\smallskip
(2) $A = xy$
\smallskip

After applying a substitution of the form $w \mapsto w + Q$ for a suitable quadratic form $Q$, we may assume that $B$ does not contain monomials divisible by $xy$.  This allows us to write 
$$B = (a_1x^3+a_2y^3)z+(a_3x+a_4y)z^3+B_0(x,y,z)^2.
$$
Note that the preimages $R_1$ and $R_2$ of $V(x)$ and $V(y)$ on $X$ are members of $|-K_X|$, hence they must be reduced.

Restricted to $V(x)$, the equation becomes
$$
w^2 + a_2y^3z + a_4yz^3,
$$
so $R_1$ is singular over $[0:\sqrt{a_4}:\sqrt{a_2}]$. Similarly, $R_2$ is singular over $[\sqrt{a_3}:0:\sqrt{a_1}]$.
Note that these points must be distinct, for otherwise $X$ is singular over $[0:0:1]$ by Proposition \ref{smooth}.

If these two points are distinct and different from $[0:0:1]$, we can apply a suitable substitution that preserves $A$ to move them to $[0:1:0]$ and $[1:0:0]$. Then, we can repeat the substitution of the first paragraph and, after rescaling, arrive at Case (a).

If the two points are distinct and one of them is $[0:0:1]$, we can assume without loss of generality that the other one lies on $V(y)$ and move it to $[1:0:0]$. After repeating the substitution of the first paragraph and rescaling, we can assume that $B_1$ is as in Case (b). Finally, after applying a substitution of the form $z \mapsto z + \lambda y, w \mapsto w + \lambda z^2 + \lambda^2 yz + \lambda^3 y^2$ for a suitable $\lambda$, we may assume that $B_0(0,1,0) = 0$.

\smallskip
 (3)  
$A = x^2$
\smallskip

Applying a substitution of the form $w \mapsto w + Q$ for a suitable quadratic form $Q$ allows us to assume that $B = xB_1 + B_0$ for homogeneous forms $B_0$ and $B_1$ in $y$ and $z$.

Let $R'$ be the preimage of $V(x)$. As in Case (2), since $R' \in |-K_X|$, $R'$ must be reduced. Restricted to $V(x)$, the double cover becomes
$$
w^2 + B_0(y,z) = 0,
$$
hence $R'$ is singular over the common zero of $B_{0,y}$ and $B_{0,z}$.
We can assume that this zero lies at $[0:0:1]$, that is, that $yz^3$ does not occur in $B_0$ and $y^3z$ occurs with non-zero coefficient. After rescaling, we may assume that $y^3z$ occurs with coefficient $1$. 

Applying a substitution of the form $z \mapsto z + \lambda_1 x + \lambda_2 y, y \mapsto y + \lambda_3 x$ for suitable $\lambda_i$ and repeating the substitution of the first paragraph, we can eliminate the monomials $y^3$ and $y^2z$ in $B_1$ and the monomial $y^4$ in $B_0$. Computing partials, we see that $X$ is singular if and only if $B_1(0,1) = 0$. Hence, after rescaling, we may assume that $B$ is as claimed. 
\end{proof}

\subsection{Fake bitangents and odd theta characteristics} \label{sec: fakebitangent}
It is known that a del Pezzo surface $X$ of degree $2$ contains $56$ $(-1)$-curves (see \cite[8.7]{CAG}, 
where the proof is characteristic free).   
They come in pairs $E_i+E_i' \in |-K_X|$ with $E_i\cdot E_i' = 2$. The Geiser involution $\gamma$ switches the two curves in a pair.
The image of each pair under any birational  morphism $\pi:X\to \bbP^2$
 is either the union of a line through two points $p_i,p_j$ and the conic through the remaining 5 points, or a cubic passing through $p_1,\ldots,p_7$ with a double point at some $p_i$ (and one curve of the pair is contracted by $\pi$).
The image of $E_i+E_i'$ under the anti-canonical map $f$ is a line $\ell$. 

If $p \neq 2$, each of the resulting $28$ lines is a bitangent line to the branch quartic curve and, 
conversely, every bitangent to the branch quartic gives rise to a pair of $(-1)$-curves. 
A bitangent line intersects the branch curve at two 
points, not necessarily distinct, whose sum is an odd theta characteristic of the curve. It is known that the number of odd theta characteristics on a smooth 
curve of genus $3$ is equal to $28$.

For arbitrary $p$, we still have the following.

\begin{lemma}\label{lm:4.1} The preimage $f^{-1}(\ell)$ of a line $\ell$ is a sum of two $(-1)$-curves if and only if $f^{-1}(\ell)$ is reducible.
\end{lemma}

\begin{proof} Since $f$ has degree $2$ and $\ell$ is integral, the curve $f^{-1}(\ell)$ is reducible if and only if it has two irreducible components $L_1$ and $L_2$. These components satisfy $L_1+L_2\in |-K_X|$, $L_1\cdot L_2 = 2$, and $L_1^2 = L_2^2$. Via adjunction, this easily implies that 
$L_1$ and $L_2$ are $(-1)$-curves. The converse is clear.
\end{proof}

So, even if $p = 2$, we have $28$ splitting lines, which we call \emph{fake bitangent lines} in analogy with the situation in the other characteristics. For the rest of this section, we assume $p = 2$. Since the anti-canonical map is 
\'etale outside the branch curve $V(A)$, the intersection  $E_i\cap E_i'$ lies on the ramification curve 
$R$.   Let 
$\calL = \calO_R(E_i) \cong \calO_R(E_i')$. It is an invertible sheaf on $R$ of degree 2. 
We have 
$$\calL^{\otimes 2}\cong \calO_R(E_i+E_i')\cong \calO_R(-K_X).$$
Since $B\in |\calO_{\bbP^2}(2)|$, we have $R\in |-2K_X|$. By the adjunction formula
$$\omega_R \cong \calO_R(-2K_X+K_X) \cong \calL^{\otimes 2}.$$
Invertible sheaves $\mathcal{L}$ on $R$ that satisfy this property are called \emph{invertible theta characteristics}.
They are called \emph{even}, \emph{odd}, or \emph{vanishing} according to whether their space of global sections is even-dimensional, odd-dimensional, or at least two-dimensional, respectively.
We note that, on singular curves, there can be theta characteristics which are not invertible, see 
\cite{Barth}, \cite{Beauville77}. In the following, we will only discuss invertible theta characteristics, so we drop the ``invertible'' from the notation.

Let $\Theta(R)$ be the set of isomorphism classes of theta characteristics on $R$ and let $J(R)$ be the identity component of the Picard scheme of $R$, also called the generalized Jacobian of $R$.
\begin{lemma} \label{lem: generalizedjacobian}
The generalized Jacobian $J(R)$ of $R$ is isomorphic to $\mathbb{G}_a^3$.
\end{lemma}
\begin{proof}
Since $R$ is of arithmetic genus $3$, $J(R)$ is a commutative group scheme of dimension $3$.
As $R_{\red}$ has only unibranched singularities, \cite[Proposition 5, Proposition 9]{NeronModels}
 shows that $J(R)$ is unipotent. Finally, we have a factorization of the absolute 
 Frobenius $F: R \to V(A) \to R$. Note that $J(V(A))$ is trivial, even if $V(A)$ is non-reduced, 
 since $H^1(V(A),\mathcal{O}_{V(A)}) = 0$. Since $F^*$ is multiplication by $p$ on $J(R)$, 
 we obtain that $J(R)$ is $p$-torsion, hence isomorphic to $\mathbb{G}_a^3$.
\end{proof}
In particular, $J(R)(k)$ is an infinite $2$-torsion group and it acts on $\Theta(R)$ via tensor products. 
It is easy to check that $\Theta(R)$ is a torsor under 
$J(R)(k)$ via this action. This already shows that the problem of finding (fake) bitangents 
using theta characteristics on $R$ in characteristic $2$ is much more subtle than it is in the other
 characteristics. Let us give an example that further illustrates this point.
\begin{example} \label{ex: thetacharacteristic}
Assume that $V(A)$ is a smooth conic.

Consider $\pi: R \to V(A) \overset{\sim}{\to} \mathbb{P}^1$. We have $\pi^* \mathcal{O}_{\mathbb{P}^1}(2) = (f|_R)^* \mathcal{O}_{V(A)}(1) = (\omega_X)|_R$, so $\mathcal{L}  \coloneqq \pi^* \mathcal{O}_{\mathbb{P}^1}(1)$ is a theta characteristic on $R$. Moreover, we have $h^0(R,\pi^* \mathcal{O}_{\mathbb{P}^1}(1)) = 2$, so $\mathcal{L}$ is a vanishing theta characteristic. In fact, this is the unique vanishing theta characteristic on $R$: Indeed, let $\mathcal{L}'$ be another vanishing theta characteristic. Then, the Riemann--Roch formula yields
$$
h^0(R,\mathcal{L} \otimes \mathcal{L}') - h^0(R,\omega_R \otimes \mathcal{L}^{-1} \otimes \mathcal{L}'^{-1}) = 2.
$$
Since $h^0(R,\mathcal{L}) \geq 2$ and $h^0(R,\mathcal{L}') \geq 2$, we have $h^0(R,\mathcal{L} \otimes \mathcal{L}') \geq 3$, so $h^0(R,\omega_R \otimes \mathcal{L}^{-1} \otimes \mathcal{L}'^{-1}) \neq 0$. Since $R$ is integral and $\omega_R \otimes \mathcal{L}^{-1} \otimes \mathcal{L}'^{-1}$ has degree $0$, this implies that $\mathcal{L} \cong \mathcal{L}'$.

Next, let $\ell$ be any line in $\mathbb{P}^2$ such that $f^{-1}(\ell)$ meets $R$ in two distinct smooth points. Then, $f^{-1}(\ell \cap V(A))_{\red}$ defines an effective theta characteristic $\mathcal{L}$ on $R$. By the previous paragraph, we have $h^0(R,\mathcal{L}) = 1$, hence all the infinitely many theta characteristics arising in this way are odd. It would be interesting to find an abstract characterization of the fake bitangent lines among the odd theta characteristics of $R$.
\end{example}

Nevertheless, we can find explicit equations of fake bitangent lines using the following result.

\begin{lemma} \label{lem: splittingcriterion}
Let $C \to \bbP^1$ be an Artin--Schreier double cover given by an equation of the form
$$w^2+f(u,v)w+g(u,v) = 0,$$
 where $f$ and $g$ are homogeneous polynomials of degree $n$ and $2n$, respectively, and $f \neq 0$. 
 Then, $C$ is reducible if and only if there exists a homogeneous polynomial $h$ of degree $n$ with $g(u,v) = f(u,v)h(u,v)+h(u,v)^2$.
\end{lemma}

\begin{proof} 
If there exists an $h$ as in the assertion, then $w^2 + fw + g = (w + f + h)(w + h)$, so $C$ is obviously reducible.

Conversely, assume that $C$ is reducible. Then, $C$ has exactly two irreducible components and these components are interchanged by the substitution $w \mapsto w + f$.
In other words, we can write $w^2 + fw + g = h' (h' + f)$, where $h'$ is a weighted homogeneous polynomial of degree $n$. This is only possible if $h'$ is of the form $h' = w + h$ for some $h$ homogeneous of degree $n$ in the variables $u$ and $v$. Then, $w^2 + fw + g = (w +h)(w + h + f) = w^2 + fw + h^2 + fh$, hence $g = fh + h^2$, as claimed. 
\end{proof}

Finally, for later use, we record some simple restrictions on the possible positions of fake bitangent lines with respect to the singularities of $R$.

\begin{proposition}\label{prop:strangepoint} 
Let $\ell$ be a fake bitangent line that passes through the image $P$ of a singular point of an irreducible component of $R_{\red}$. Then, $V(A)$ is smooth and $\ell$ is tangent to $V(A)$ at $P$.
\end{proposition}

\begin{proof} 
Write $f^{-1}(\ell) = L_1 + L_2$. Since $R_{\red}$ is singular at $f^{-1}(P)$, $L_i$ and $R$ have intersection multiplicity at least $2$ in $f^{-1}(P)$. Since $R \in |-2K_X|$ and $L_1 + L_2 \in |-K_X|$, we have $(L_1 + L_2).R = 2K_X^2 = 4$, hence $L_1 + L_2$ and $R$ meet only in $f^{-1}(P)$. Therefore, their images in $\mathbb{P}^2$ meet only in $P$. If $V(A)$ is smooth, this implies that $\ell$ is tangent to $V(A)$ in $P$. If $V(A)$ is the union of two lines, this implies that $\ell$ passes through their intersection. However, in this case, $L_i$ and $R$ have intersection multiplicity at least $3$ in $f^{-1}(P)$, which is absurd.
Finally, if $V(A)$ is a double line, then $R_{\red} \in |-K_X|$ and $2 = K_X^2 = (L_1 + L_2).R_{\red} \geq 4$, 
a contradiction.
\end{proof}

\begin{remark}
We note that there are del Pezzo surfaces for which fake bitangents satisfying the properties of Proposition \ref{prop:strangepoint} exist. See Proposition \ref{prop: strangefibration} for a classification in terms of the normal forms of Theorem \ref{canform}.
\end{remark}

\subsection{Strange elliptic fibrations} \label{sec: strangefibration}
To each del Pezzo surface $X$ of degree $2$ in characteristic $2$ with branch locus $V(A)$ of the anti-canonical map $f: X \to \bbP^2$, there is a naturally associated point $P_X$ in $\bbP^2$: If $V(A)$ is smooth, we let $P_X$ be the strange point of $V(A)$, if $V(A)$ is the union of two lines, we let $P_X$ be their intersection, and if $V(A)$ is a double line, we let $P_X$ be the image of the singular point of $f^{-1}(V(A))_{\red}$. We call $P_X$ the \emph{strange point} of $X$ and note that the action of $\Aut(X)$ fixes $P_X$.

The pencil $\calP$ of lines through $P_X$ is $\Aut(X)$-invariant as well. Its preimage $\calC$ in $X$ is an $\Aut(X)$-invariant pencil of curves of arithmetic genus $1$ with two base points if $V(A)$ is smooth and with one base point of multiplicity $2$ if $V(A)$ is singular. 
We let $\pi:Y \to X$ be the blow-up of the base points of $\calC$. Then, 
$\calC$ defines a relatively minimal genus one fibration $\phi: Y \to \bbP^1$. Since the map $X\to \bbP^2$ is separable and 
a general line in the pencil is not contained in $V(A)$, its pre-image on $Y$ is a smooth elliptic curve (see the details in the proof of the next Proposition). Thus, the genus on fibration 
is an elliptic fibration. We  call it the \emph{strange elliptic fibration} associated to $X$.

By construction, the group $\Aut(X)$ lifts to a subgroup 
$\Aut(Y)$ and we will use this in Proposition \ref{prop: strangeaut} to find restrictions on the 
possible structure of $\Aut(X)$. To make the most of this connection, we will now describe the singular fibers of the elliptic fibration $\phi:Y\to \bbP^1$. 
We employ Kodaira's notation: we say that a fiber isomorphic to an irreducible cuspidal cubic curve is 
of type ${\rm II}$,  a fiber that 
consists of two smooth rational curves intersecting non-transversally at one point is of type ${\rm III}$, and a
fiber that consists of three smooth rational curves intersecting at one point is of type ${\rm IV}$.

We use the normal forms of Theorem \ref{canform}, so that $A=x^2+yz,xy,$ or $x^2$ and $P_X = [1:0:0]$ in the first case and $P_X = [0:0:1]$ in the other two cases.
In the first case, we let $\ell_{[t_0:t_1]}$ be the line $V(t_0y + t_1z)$ and in the other two cases, we let $\ell_{[t_0:t_1]}$ be the line $V(t_0x + t_1y)$. The fiber of $\phi$ corresponding to $\ell_{[t_0:t_1]}$ is denoted by $F_{[t_0:t_1]}$.

\begin{proposition}\label{prop: strangefibration2}
The generic fiber of the strange elliptic fibration associated to $X$ is a supersingular elliptic curve. Its singular fibers are of type ${\rm II}$, ${\rm III}$, or ${\rm IV}$ and its Mordell-Weil group is torsion-free. More precisely:
\begin{enumerate}
\item[(1) \phantom{(a)}]  If $A = x^2+yz$, then the following hold:
\begin{itemize}
    \item The fiber $F_{[t_0:t_1]}$ is smooth if and only if $t_0y + t_1z \nmid B_1$. 
    \item The fiber $F_{[t_0:t_1]}$ is of type ${\rm III}$ if $\ell_{[t_0:t_1]}$ is a fake bitangent and of type ${\rm II}$ otherwise.
    \item The line $\ell_{[1:0]}$ is a fake bitangent if and only if $e = 0$.
    \item The line $\ell_{[0:1]}$ is a fake bitangent if and only if $a = 0$.
    \item The line $\ell_{[1:1]}$ is a fake bitangent if and only if  $a+b+c+d+e = 0$.
\end{itemize}
\item[(2) (a)] If $A = xy$ and $B_1 = xz^3 + yz^3$, then the following hold:
\begin{itemize}
\item The fiber $F_{[t_0:t_1]}$ is smooth if and only if $[t_0:t_1]\ne [1:0], [0:1], [1:1]$. 
\item $F_{[1:0]}$ and $F_{[0:1]}$ are of type ${\rm II}$.
\item $F_{[1:1]}$ is of type ${\rm IV}$ if $\ell_{[1:1]}$ is a fake bitangent and of type ${\rm III}$ otherwise.
\item The curve $\ell_{[1:1]}$ is a fake bitangent if and only if $c = d^2 + e^2$.
\end{itemize}
\item[(2) (b)] If $A = xy$ and $B_1 = xz^3 + y^3z$, then the following hold:
\begin{itemize}
\item The fiber $F_{[t_0:t_1]}$ is smooth if and only if $[t_0:t_1]\ne [1:0], [0:1]$. 
\item The curve $F_{[0:1]}$ is of type ${\rm II}$.
\item The curve $F_{[1:0]}$ is of type ${\rm III}$. 
\end{itemize}
\item[(3) \phantom{(a)}] If $A = x^2$, then the following hold:
\begin{itemize}
    \item The fiber $F_{[t_0:t_1]}$ is smooth if and only if $[t_0:t_1]\ne [1:0]$.
    \item The curve $F_{[1:0]}$ is of type ${\rm III}$.
\end{itemize}
\end{enumerate}
\end{proposition}

\begin{proof}
We study each case separately. The Mordell-Weil group is torsion free by \cite[Main Theorem]{OguisoShioda}, since the lattice spanned by fiber components is of rank at most $4$ in each case.

\begin{enumerate}
    \item[(1) \phantom{(a)}] In this case $A = x^2 + yz$. 
    
    First, consider $\ell_{[1:t]} = V(y + tz)$. Plugging $y = tz$ into the equation of $X$, we obtain
$$
w^2 + (x^2 + tz^2)w + xB_1(tz,z) + (at^4 + bt^3 + ct^2 + dt + e)z^4
$$
with $B_1(tz,z) \in \{ \lambda t(t+1)z^3,t^2z^3,t^3z^3\}$.
If $y + tz \nmid B_1$, then $B_1(tz,z) \neq 0$, so taking partials with respect to $x$ and $w$ shows that a singular point must satisfy $x = z = 0$, which is absurd.
If $y + tz \mid B_1$, then $B_1(tz,z) = 0$ and $F_{[1:t]}$ is singular over $[t:t:1]$. 
Similarly, one checks that $F_{[0:1]}$ is singular.

The equation 
$$
w^2 + (x^2 + tz^2)w + xB_1(tz,z) + (at^4 + bt^3 + ct^2 + d + e)z^4
$$
shows that $F_{[1:t]}$ is a double cover of $\bbP^1$ branched 
over a single point. Hence, if $F_{[1:t]}$ is smooth, then it is supersingular, and 
if it is singular and irreducible, it is a cuspidal cubic.

Finally, consider the curve $F_{1:0}$ given by
$$
w^2 + x^2w + ez^4.
$$
By Lemma \ref{lem: splittingcriterion}, it is clear that $F_{[1:0]}$ is reducible if and only if $e = 0$.
The calculation for $F_{[1:1]}$ and $F_{[0:1]}$ is similar.

\item[(2) (a)] In this case $A = xy$ and $B_1 = xz^3 + yz^3$.

First, consider $\ell_{[1:t]} = V(x + ty)$ with $t \neq 0,1$. Plugging $x = ty$ into the equation of $X$, we obtain
$$
w^2 + ty^2w + (t+1)yz^3 + B_0(ty,y,z)^2.
$$
Then, taking partials shows that $F_{[1:t]}$ is smooth. Since it is a double cover of $\bbP^1$ branched over a single point, it is supersingular.

Next, consider $F_{[1:1]}$, whose image in $X$ is given by
$$
w^2 + y^2w + ((a + b)y^2 + (d + e)yz + cz^2)^2.
$$
This curve is singular over $[0:0:1]$, so $F_{[1:1]}$ has one irreducible component contracted by $Y \to X$. By Lemma \ref{lem: splittingcriterion}, the image of $F_{[1:1]}$ in $X$ is reducible if and only if $c = d^2 + e^2$.

Finally, the curves $F_{[1:0]}$ and $F_{[0:1]}$ are isomorphic to their images in $X$ and these images are irreducible and cuspidal by Theorem \ref{canform}.

\item[(2) (b)] In this case $A = xy$ and $B_1 = xz^3 + y^3z$.

First, consider $\ell_{[1:t]} = V(x + ty)$ with $t \neq 0$. Plugging $x = ty$ into the equation of $X$, we obtain
$$
w^2 + ty^2w + tyz^3 + y^3z + B_0(ty,y,z)^2.
$$
As in the previous cases, taking partials shows that $F_{[1:t]}$ is smooth and supersingular.

The curve $F_{[0:1]}$ is isomorphic to its image in $X$, since $f^{-1}(\ell_{[0:1]})$ is smooth over the point $[0:0:1]$. Hence, $F_{[0:1]}$ is of type ${\rm II}$. On the other hand, the curve $F_{[1:0]}$ is of type ${\rm III}$, since its image in $X$ has multiplicity $2$ over $[0:0:1]$.

\item[(3) \phantom{(a)}] In this case $A = x^2$.

First, consider $\ell_{[t:1]} = V(tx + y)$. Plugging $y = tx$ into the equation of $X$, we obtain
$$
w^2 + x^2w + xz^3 + (b t^2 + at)x^2z^2 + c z^4.
$$
Then, taking partials shows that $F_{[t:1]}$ is smooth. Since it is a double cover of 
$\bbP^1$ branched over a single point, it is supersingular.

The curve $F_{[1:0]}$ is of type ${\rm III}$, by the same argument as in the previous case.
\end{enumerate}
\end{proof}

\begin{remark}
The classification of singular fibers of rational elliptic surfaces with a section in characteristic $2$ can be found in \cite{Lang}.
Lang shows that in the cases where the general fiber is a supersingular elliptic curve, the 
number of singular fibers is at most $3$, which agrees with what we observed in the case of strange elliptic fibrations. Proposition \ref{prop: strangefibration2} shows that the singular fibers that occur on strange genus one 
fibrations are of type $9A,9B,10A,10B,10C$ or $11$ in Lang's terminology.
\end{remark}

\section{Automorphism groups of del Pezzo surfaces of degree $2$}

\subsection{Preliminaries} Recall once more from Section \ref{SS: 2.1} that a del Pezzo surface $X$ of degree 2 is a 
surface of degree $4$ in $\bbP(1,1,1,2)$ given by an equation of the form
$$
w^2 + A(x,y,z)w + B(x,y,z) = 0.
$$
Since this is the anti-canonical model of $X$ and $\omega_X^{-n}$ admits a natural $\Aut(X)$-linearization for all $n$, we obtain that $\Aut(X)$ is isomorphic to the subgroup of $\Aut(\bbP(1,1,1,2))$ of automorphisms that preserve $X$.

The structure of the group $\Aut(\bbP(1,1,1,2))$ is well-known. The vector space $\Bbbk[x,y,z]_2$ of quadratic forms is 
a normal subgroup of $\Aut(\bbP(1,1,1,2))$ that acts via $(x,y,z,w) \mapsto (x,y,z,w+Q)$. The quotient by this subgroup 
is the group of transformations that change $(x,y,z)$ linearly and multiply $w$ by a scalar. Since the transformation 
$(x,y,z,w)\mapsto (\lambda x,\lambda y, \lambda z, \lambda^2w)$ is the identity, this quotient is isomorphic to 
$\GL_3(\Bbbk)/\mu_2(\Bbbk)$. Since we are in characteristic $2$, the subgroup $\mu_2(\Bbbk)$ is trivial. This gives an isomorphism
$$\Aut(\bbP(1,1,1,2))\cong \Bbbk[x,y,z]_2: \GL_3(\Bbbk).$$
We denote elements of this group by $(Q,g)\in \Bbbk[x,y,z]_2 \times \GL_3(\Bbbk)$ where the semi-direct product structure is
$$(Q,g)\circ (Q',g') = (g^*(Q')+Q,gg').$$

Using this description of $\Aut(\bbP(1,1,1,2))$, it is straightforward to calculate the subgroup of automorphisms preserving $X$. We obtain
$$\Aut(X) \cong \{(Q,g): g^*(A) = A, g^*(B) = B+AQ+Q^2\}.$$
The kernel of the homomorphism 
$$\Aut(X) \to \GL_3(\Bbbk), \quad (Q,g) \mapsto g$$
is generated by the Geiser involution $\gamma$. We let $G(X)$ be the image of $\Aut(X)$ in $\GL_3(\Bbbk)$.
 
\begin{lemma} \label{lem: injective}
The homomorphism $G(X) \to \GL_3(\Bbbk) \to \PGL_3(\Bbbk)$ is injective.
\end{lemma}
\begin{proof}
Let $g \in G(X)$ be in the kernel of this homomorphism. Then, $g = \lambda {\rm I}_3$ for some $\lambda \in \Bbbk^{\times}$. On the other hand, by definition of $G(X)$, we have $g^*(A) = A$. Since $A$ has degree $2$, this implies $\lambda^2 = 1$. Hence, $\lambda = 1$.
\end{proof}

We recall from \cite[\S 1]{DM} that a choice of a geometric basis of a blow-up $X\to \bbP^2$ of seven points defines 
an injective homomorphism
\beq
\rho:\Aut(X)\to W(\sfE_7).
\eeq
The image of the Geiser involution is equal to $-\id_{\sfE_7}.$ It is known that 
$W(\sfE_7) = \la -\id_{\sfE_7} \ra \times W(\sfE_7)^+,$  where $W(E_7)^+ \subseteq W(E_7)$ is the kernel of the
 determinant map. 
 
In particular, to determine $\Aut(X)$, it suffices to determine $G(X)$ and both groups are isomorphic to subgroups of $W(E_7)$ via $\rho$. This puts severe restrictions on the possible structure of $G(X)$. Finally, we can use the strange genus one fibrations of the previous section to get information on $G(X)$.
\begin{proposition} \label{prop: strangeaut}
Let $\phi:Y \to \bbP^1$ be the strange elliptic fibration associated to $X$. Choose an exceptional curve $E$ of $Y \to X$ as the zero section of $\phi$ and let $C$ be the second exceptional curve. Then, there is a homomorphism $\varphi: \Aut(X) \to \Aut(Y)$ that satisfies the following properties:
\begin{enumerate}
    \item $\varphi$ is injective.
    \item $\varphi(\gamma)$ preserves every fiber of $\phi$.
    \item If $V(A)$ is smooth, then $C$ is a section of $\phi$. We have $\varphi(\gamma) = t_C \circ \iota$, where $\iota$ is the negation automorphism and $t_C$ is translation by $C$.
    \item If $V(A)$ is singular, then $C$ is a component of a reducible fiber of $\phi$. 
    We have $\varphi(\gamma) = \iota$ and $\varphi$ factors through the stabilizer of the pair $(E,C)$. 
\end{enumerate}
\end{proposition}
\begin{proof}
The surface $Y$ is obtained by blowing up $X$ in two points that are uniquely determined by $V(A)$, hence stable under the action of $\Aut(X)$. This shows existence and injectivity of the homomorphism $\varphi$. The fibration $\phi$ is induced by the pencil of lines in $\bbP^2$ through the strange point of $X$. Since $\gamma$ preserves these lines, it preserves the fibers of $\phi$.

If $V(A)$ is smooth, then $E$ and $C$ are interchanged by $\varphi(\gamma)$. 
The automorphism $t_{-C} \circ \varphi(\gamma) \circ \iota$ maps $E$ to $E$ and $-C$ to $-C$. It is well-known that every fixed point of a non-trivial automorphism of an elliptic curve is a torsion point. On the other hand, by Proposition \ref{prop: strangefibration2}, $\phi$ has no torsion sections, so $t_{-C} \circ \varphi(\gamma) \circ \iota = {\rm id}$, which yields Claim (3).

If $V(A)$ is singular, then $C$ is a $(-2)$-curve which meets $E$, hence it is the identity component of a reducible fiber of $\phi$. Since $\varphi(\gamma)$ is an involution that preserves $E$, we have $\varphi(\gamma) = \iota$ and we obtain Claim (4). 
\end{proof}

\subsection{Classification}

\begin{theorem} \label{thm: main2}
Every del Pezzo surface of degree $2$ in characteristic $2$ such that in the decomposition $\Aut(X) \cong 2 \times G(X)$ the group $G(X)$ is non-trivial is a surface of degree 4 in $\bbP(1,1,1,2)$ given by an equation of the form
$$
w^2 + Aw + B,
$$
where $(A,B,G(X))$ is one of the following:
 $$
 \resizebox{\textwidth}{!}{
 $
    \begin{array}{|c|c|c|c|c|c|c|} \hline
    \text{Name}  & A & B & B_1 & B_0 & G(X) &\# \text{Parameters}  \\ \hline \hline
    (1) (a) (i)\phantom{ii} & x^2+yz & xB_1 + B_0 & \lambda yz(y+z) & ay^4 + by^3z + cy^2z^2 + byz^3 + az^4 & 2 &4\\ \hline
    (1) (a) (ii)\phantom{i} & x^2+yz & xB_1 + B_0 & \lambda yz(y+z) & ay^4 + \lambda y^3z + a y^2z^2 + \lambda yz^3 + az^4 & S_3 & 2\\ \hline
    (1) (c) (i)\phantom{ii} & x^2+yz & xB_1 + B_0 & y^3 & by^3z + cy^2z^2 + ez^4 & 2^3 & 3 \\ \hline
    (2) (a) (i)\phantom{ii} & xy & B_1 + B_0^2 & xz^3 + yz^3 & ax^2 + ay^2 + cz^2 + dxz + dyz & 2 &3  \\ \hline
    (2) (a) (ii)\phantom{i} & xy & B_1 + B_0^2 & xz^3 + yz^3 & ax^2 + by^2 & 3 & 2  \\ \hline
    (2) (a) (iii) & xy & B_1 + B_0^2 & xz^3 + yz^3 & ax^2 + ay^2 & 6 & 1  \\ \hline
    (3) (i)\phantom{(a)ii}  & x^2 & xB_1 + B_0 & z^3 & y^3z + cz^4 & 3 & 1 \\ \hline
    (3) (ii)\phantom{(a)i} & x^2 & xB_1 + B_0 & z^3 & y^3z & 9 & 0 \\ \hline
    \end{array}
    $
    }
    $$
The parameters satisfy the following conditions:
\begin{enumerate} [leftmargin=21mm]
    \item[(1) (a) (i)\phantom{ii}] $\lambda \neq 0, \lambda^2 + c \neq 0, b^2 + a \neq 0, (b,c) \neq (\lambda,a)$.
    \item[(1) (a) (ii)\phantom{i}] $\lambda \neq 0, \lambda^2 + a \neq 0.$
    \item[(1) (c) (i)\phantom{ii}] $e \neq 0$.
    \item[(2) (a) (i)\phantom{ii}] $a \neq 0, (c,d) \neq (0,0)$.
    \item[(2) (a) (ii)\phantom{i}] $a \neq 0, b \neq 0, a \neq b$.
    \item[(2) (a) (iii)] $a \neq 0$.
    \item[(3) (i) \phantom{(a)ii}] $c \neq 0$.
    \item[(3) (ii)\phantom{(a)ii}] None. 
\end{enumerate}
\end{theorem}

\begin{proof}
We use the normal forms of Theorem \ref{canform} and the description of $\Aut(X)$ and $G(X)$ given in 
the beginning of the current section. We go through the cases of Theorem \ref{thm: main2}.

\begin{enumerate}
    \item[(1) (a)] Here, $X$ is given by an equation of the form
        $$
        w^2 + (x^2 + yz)w + \lambda xyz(y+z) +  B_0
        $$
        with
        $$
        B_0 = ay^4 + by^3z + cy^2z^2 + dyz^3 + ez^4
        $$
        and the cusps lie over $[0:1:0],[0:0:1]$, and $[1:1:1]$. Let $(Q,g) \in \Aut(X)$ 
        be an automorphism of $X$. Then, $g$ preserves the three points lying under the cusps. 
        Moreover, if $g$ fixes the three cusps, then it fixes $V(A)$ pointwise, hence $g$ is trivial 
        in $\PGL_3(\Bbbk)$, so, by Lemma \ref{lem: injective}, $g$ is the identity and $(Q,g)$ coincides with the Geiser involution. 
        Hence, $G(X)$ acts faithfully on $\{[0:1:0],[0:0:1],[1:1:1]\}$.
        
        Note that $G(X)$ contains an involution if and only if $X$ admits an equation where this involution 
        is given by $y \leftrightarrow z$. This involution is in $G(X)$ if and only 
        if there exists a quadratic form $Q$ such that
        \begin{equation}\label{eq: Case1a}
        Q^2 + (x^2 + yz)Q  = g^*B_0 + B_0.
        \end{equation}
        Since $Q^2 + (x^2 + yz)Q$ contains a non-zero monomial divisible by $x^2$ as soon as it is 
        non-zero and $g^*B_0 + B_0$ does not contain such a monomial, 
        we must have $Q \in \{0, x^2 + yz\}$ and Equation \eqref{eq: Case1a} holds if and only if 
        $a = e$ and $b = d$, as claimed.

       Next, note that $G(X)$ contains an automorphism $g$ of order $3$ if and only if $g$ is given by
       $x \mapsto x + z, y \mapsto z, z \mapsto y + z$ and 
       there exists a quadratic form $Q$ such that
       \begin{equation} \label{eq: Case1b}
        Q^2 + (x^2 + yz)Q  =  \lambda yz^2(y+z) + g^*B_0 + B_0.
        \end{equation}
        By the same argument as in the previous paragraph, we have $Q \in \{0, x^2 + yz\}$ and 
        Equation \eqref{eq: Case1b} holds if and only if $a = c = e$ and $b = d = \lambda$. In particular, note that these conditions imply the conditions of the previous paragraph in this case, hence $G(X) = S_3$.

        \item[(1) (b)] In this case, $V(A)$ is smooth and $R$ has two non-isomorphic singularities. Then, $g\in G(X)$ must fix the images 
        of them on $V(A)$. Since an automorphism of order 2 of $\bbP^1$ has only one fixed point, we may assume 
        that the order of $g$ is odd. 
        By Proposition \ref{prop:strangepoint}, the line $\ell$ through the images of the singularities is not a fake bitangent. Its preimage $E$ in $X$ is an integral curve of arithmetic genus one and the Geiser involution has two fixed points on $E$. Hence, either $E$ is smooth and ordinary, or nodal. In both cases, there is no non-trivial automorphism of odd order that commutes with the involution, hence $g$ fixes $\ell$ pointwise. Since $g$ also fixes the strange point $P$ on $V(A)$ and the projection from $P$ is inseparable, $g$ fixes $V(A)$ pointwise, hence $g$ is the identity.
    
         \item[(1) (c)] Here, $X$ is given by an equation of the form
        $$
        w^2 + (x^2 + yz)w + xy^3 + B_0
        $$
        with
        $$
        B_0 = by^3z + cy^2z^2 + dyz^3 + ez^4.
        $$
        The singularity of $R$ lies over $[0:0:1]$. 
        An element $g \in G(X)$ of odd order has at least two fixed points on $V(A)$ and then the same argument as in the previous case shows that $g$ is the identity.
        Therefore, $G(X)$ is a $2$-group that acts on $V(A) \cong \bbP^1$ with a fixed point. In particular,
        $G(X)$ is isomorphic to a subgroup of $\mathbb{G}_a(\Bbbk)$, hence isomorphic to 
         $2^n$ for some $n \geq 0$. 
         
          We may assume that $g$ acts as $x\mapsto x+\alpha y, y\mapsto y, z\mapsto z+\alpha^2y$.
         Then $g$ lifts to $\Aut(X)$ if and only if there exists a quadratic form $Q$ such that 
         $$(x^2+yz)Q+Q^2 = \alpha y^4+ g^*(B_0)+B_0. $$
         Since the right-hand side does not contain a monomial divisible by $x^2$, we get, 
         as in the previous cases, $Q = x^2+yz$ or $Q= 0$. Comparing coefficients yields the system of equations 
         \begin{eqnarray*}
        d\alpha^2 &=& 0 \\
        d\alpha^4&=& 0 \\
       e\alpha^8+d\alpha^6+c\alpha^4 + b \alpha^2+\alpha &=& 0 \\
        \end{eqnarray*}
        So, if $d \neq 0$, then $\alpha = 0$ and $G(X)$ is trivial. If $d = 0$, there are $8$ possibilities for $\alpha$, one for each root of $ex^8 + cx^4 + bx^2 + x$. All the roots are distinct since the derivative of this polynomial is $1$. Here, we also use that $e \ne 0$ by Theorem \ref{canform}. Thus
         $G(X) \cong 2^3$.

        \item[(2) (a)] Here, $X$ is given by an equation of the form
        $$
        w^2 + xyw + xz^3 +yz^3 + B_0^2
        $$
        with
        $$
        B_0 = ax^2 + by^2 + cz^2 + dxz + eyz.
        $$
        The singularities of the irreducible components of $R$ lie over $[1:0:0]$ and $[0:1:0]$.
        Let $(Q,g) \in \Aut(X)$. Then, $g$ preserves these two points and the intersection of $V(x)$ and $V(y)$. 
        Moreover, by Proposition \ref{prop: strangeaut} and Proposition \ref{prop: strangefibration2}, $g$ preserves the line $V(x+y)$.
        
        Assume that $g$ has odd order. Then, $g$ preserves the three lines $V(x),V(y)$, and $V(x+y)$, hence it is of the form $(x,y,z) \mapsto (x, y, \alpha z)$. The quadratic form $Q$ satisfies
        $$
        Q^2 + xyQ = g^*(B_1 + B_0^2) + B_1 + B_0^2.
        $$
        The right-hand side does not contain monomials divisible by $xy$, hence $Q \in \{0,xy\}$.
        Now, $g^*B_1 + B_1 = 0$ implies that $\alpha^3 = 1$, and if $\alpha \neq 1$, then $g^*B_0^2 + B_0^2 = 0$ holds if and only if $c = d = e = 0$.
        
        Assume that $g$ has order a power of $2$.
        If $g$ does not swap the points $[1:0:0]$ and $[0:1:0]$, then it acts diagonally, hence it is the identity. Therefore, we may assume that $g$ swaps these two points and $g^2 = {\rm id}$. Hence, $g$ acts as $x \leftrightarrow y$. The quadratic form $Q$ satisfies
        $$
        Q^2 + xyQ = g^*(B_1 + B_0^2) + B_1 + B_0^2,
        $$
        hence $a = b$ and $d = e$.

        \item[(2) (b)] Here, $X$ is given by an equation of the form
        $$
        w^2 + xyw + xz^3 +y^3z + B_0^2
        $$
        with
        $$
        B_0 = ax^2 + cz^2 + dxz + eyz
        $$
        and the singularities of the irreducible components of $R$ map to $[1:0:0]$ and $[0:0:1]$. Let $(Q,g) \in \Aut(X)$.
        
        If $g$ has odd order, then there is a $g$-invariant line $\ell$ through $[1:0:0]$ and we may assume that $\ell \not \subseteq V(A)$. By the same argument as in Case (1) (b), $\ell$ is fixed pointwise.
        Then, every line through $[0:0:1]$ is $g$-invariant. Since $g^*A = A$, this means that $g$ acts as $(x,y,z) \mapsto (x,y,\alpha z)$. An automorphism of this form satisfies $g^*B_1 = B_1$ if and only if $\alpha = 1$, so $g$ is trivial.
        
        If $g$ has order a power of $2$, then by Proposition \ref{prop: strangefibration2}, $\varphi((Q,g)) \in \Aut(Y)$ preserves the two singular fibers of $\phi: Y \to \bbP^1$, hence $\varphi((Q,g))$ acts trivially on the base of $\phi$. The $2$-Sylow subgroup of automorphisms of the geometric generic fiber of $\phi$ is the quaternion group $Q_8$ and $\varphi((Q,g))$ commutes with the unique involution $\varphi(\gamma)$ in $Q_8$. This implies that $(Q,g) \in \langle \gamma \rangle$, so $g$ is trivial.

    \item[(3) \phantom{(a)}] If $V(A)$ is a double line, then $X$ is given by an equation of the form
    $$
    w^2 + x^2w + xB_1 + B_0
    $$
    with
    $
    B_1 = z^3 + ayz^2
    $
    and 
    $
    B_0 = y^3z + by^2z^2 + cz^4.
    $
    The singularity of $R_{\red}$ lies over $[0:0:1]$. Let $(Q,g) \in \Aut(X)$. 
    Then, $g$ is of the form
    $$(x,y,z) \mapsto (x,\alpha x+\beta y, \gamma x+\delta y +\epsilon z)$$
    with $\beta,\epsilon \neq 0$ and $Q$ satisfies the equation
   \begin{equation} \label{eq: Case3}
    Q^2 + x^2Q = x(g^*B_1 + B_1) + g^*B_0 + B_0.
    \end{equation}
    
    The monomials $y^3z,xz^3,xyz^2,xy^2z$ and $xy^3$ do not appear on the left-hand side, hence their coefficients on the right-hand side must be zero. This yields the following conditions:
    \begin{eqnarray*}
    \epsilon &=& \beta^{-3} \\
    \beta^9 &=& 1 \\
    \delta &=& a (\beta + \beta^6) \\
    \alpha &=& a^2(1 + \beta ) \\
    \gamma &=& a^3(1 + \beta^2)
    \end{eqnarray*}
    So, the order of $g$ is equal to the order of $\beta$ in $k^{\times}$, hence it is equal to $1$, $3$ or $9$. Now, we calculate that if $\beta^3 = 1$, then $g^3$ acts as
   $$
   (x,y,z) \mapsto (x,y, a^3(\beta + \beta^2)x + z).
   $$
    Hence, if $a \neq 0$, then $g$ is the identity.
    
    So, assume that $a = 0$, so that, in particular, $\alpha = \delta = \gamma = 0$. Equation \eqref{eq: Case3} becomes
    $$
    Q^2 + x^2Q = (\epsilon^3 + 1)xz^3 + (\beta^3\epsilon + 1)y^3z + b(\beta^2\epsilon^2 + 1)y^2z^2 + c(\epsilon^4 + 1)z^4.
    $$
    On the left-hand side, the coefficients of $z^4$ and $y^2z^2$ are the squares of the coefficients of $x^2z^2$ and $x^2yz$, respectively. Since the latter monomials do not appear on the right-hand side, the coefficients of the former monomials must vanish. Therefore, we get the four conditions:
    \begin{eqnarray*}
    \epsilon^3 + 1 &=& 0 \\
    \beta^3\epsilon + 1 &=& 0 \\
    b(\beta^2\epsilon^2 + 1) &=& 0 \\
    c(\epsilon^4 + 1) &=& 0
    \end{eqnarray*}
    Hence, if $b \neq 0$, then $\beta = \epsilon = 1$, so $G(X)$ is trivial. If $b = 0$ and $c \neq 0$, then $\epsilon = 1$ and $\beta^3 = 1$, and so $G(X) \cong C_3$. If $b = c = 0$, then $\epsilon = \beta^{-3}$ and $\beta^9 = 1$, hence $G(X) \cong C_9$. 
  \end{enumerate}
\end{proof}

\begin{remark}
With our choice of normal form in Theorem \ref{thm: main2}, the map $g \mapsto (0,g)$ defines an explicit section of the surjection $\Aut(X) \to G(X)$ in every case.
\end{remark}

\begin{remark}  \label{rem: order18}
The group $2^4$ that appears in Theorem \ref{thm: main2} occurs as a group of 
automorphisms of a del Pezzo surface of degree 
$4$ in all characteristics \cite{DD}. In characteristic $0$, 
there is a unique conjugacy class of subgroups isomorphic to $2^4$ in the Cremona group.
 One can prove, using the theory of birational links, that in characteristic $2$, 
the two subgroups of $\text{Cr}_\Bbbk(2)$ are not conjugate.
\end{remark} 
 
\begin{remark} The fact that $2$ and $3$ are the only primes that divide the order of $\Aut(X)$ can be proven without the 
classification. It is known that $2,3,5$, and $7$ are the only primes that divide the order of $W(E_7)$. To exclude 
the primes $5$ and $7$, one can use the Lefschetz fixed-point formula and the known traces of elements of $W(E_7)$ 
acting on the the root lattice of type $E_7$ to get a contradiction with the possible structure of the set of fixed points of an 
element of the group $G(X)$. 
\end{remark}

\newpage

\subsection{Conjugacy classes and comparison with the classification in characteristic $0$}
In this section, we determine the conjugacy classes in $W(\sfE_7)$ of the elements of the groups that occur in Theorem \ref{thm: main2} and, whenever possible, compare the surfaces in Theorem \ref{thm: main2} with their counterparts in characteristic $0$ (see \cite[Table 8.9]{CAG}).

First, we note that, as $H^0(X,T_X) = 0$ for all del Pezzo surfaces of degree $2$, the automorphism group of any lift of a del Pezzo surface $X$ in Theorem \ref{thm: main2} to characteristic $0$ is a subgroup of $\Aut(X)$. By Theorem \ref{thm: main2}, $|\Aut(X)| \leq 18$, so Types I,$\hdots$,V and of \cite[Table 8.9]{CAG} do not have a reduction modulo $2$ which is a del Pezzo surface. Similarly, Type VII  of \cite[Table 8.9]{CAG} has no analogue in characteristic $2$.

The surface of Type $(3) (ii)$ of Theorem \ref{thm: main2} is a reduction modulo $2$ of the surface of Type VI in \cite[Table 8.9]{CAG}. Hence, we call this surface Type VI. Since the conjugacy classes of elements of $\Aut(X)$ are the same as the ones of the lift, the entry for Type VI in Table \ref{tbl:autodp2} is the same as the on in \cite[Table 7]{DM}.

The equations of the surfaces of Type $(2) (a) (iii)$ of Theorem \ref{thm: main2} define smooth surfaces in characteristic $0$ and the automorphisms $x \leftrightarrow y$ and $z \mapsto \zeta_3 z$ make sense in characteristic $0$. Hence, these surfaces lift to characteristic $0$ as del Pezzo surfaces with an action of $2 \times 6$. As explained above, del Pezzo surfaces of degree $2$ with an automorphism group of order bigger than $18$ do not have a smooth reduction modulo $2$, hence these lifts are of Type VIII \cite[Table 8.9]{CAG}, so we also call the surfaces of Type $(2) (a) (iii)$ Type VIII. As in the previous case, the conjugacy classes are the same as in \cite[Table 7]{DM}.

As for the surfaces of Type $(1) (a) (ii)$, we rewrite their equations using the substitution $x \mapsto x + y + z$ as
$$
w^2 + (x^2 + y^2 + z^2 - yz)w + \lambda xyz(z-y) + a(y^2 + z^2 - yz)^2.
$$
This equation defines a lift of $X$ to characteristic $0$ and the $\Aut(X)$-action lifts as well, since it is generated by the Geiser involution $\gamma: w \mapsto -w$, the involution $y \leftrightarrow z$ and the automorphism $g$ of order $3$ given by
\begin{eqnarray*}
x & \mapsto & -x \\
y &\mapsto & z \\
z &\mapsto & z - y.
\end{eqnarray*}
Hence, all surfaces of Type $(1) (a) (ii)$ are reductions modulo $2$ of surfaces of Type IX in \cite[Table 8.9]{CAG}. In particular, we can read off the conjugacy classes of elements of $\Aut(X)$ from \cite[Table 7]{DM}.

The surfaces of Type $(1) (c) (i)$ are the characteristic $2$ analogues of Type X from \cite[Table 8.9]{CAG}. We claim that every involution on a surface $X$ of type $(1) (c) (i)$ which is different from the Geiser involution is of conjugacy class $3A_1/4A_1$. It suffices to check this for the surface given by
$$
w^2 + (x^2+yz)w + xy^3 + z^4,
$$
where $G(X)$ acts as $g_{\alpha}: x \mapsto x + \alpha y, z \mapsto \alpha^2 x + z$ with $\alpha^8 = \alpha$. After using the substitution $z \mapsto \alpha x + y + z, y \mapsto \alpha^6x + \alpha^6 y$, the equation of $X$ becomes
$$
w^2 + (x^2 +xy +y^2 + \alpha^6(y+x)z + \alpha^4(x^2+y^2))w + \alpha^4 (x^3y + x^2y^2 + xy^3) + \alpha^3(x^4 + y^4 + z^4)
$$
and the involution $g_{\alpha}$ acts as $x \leftrightarrow y$. Then, the above equation makes sense in characteristic $0$ and defines a lift of $X$ together with the involution $g_{\alpha}$. In particular, by \cite[Table 7]{DM}, the conjugacy class of $g_{\alpha}$ is $3A_1$ or $4A_1$.

The equations of Types $(2) (a) (ii)$ and $(3) (i)$ make sense in characteristic $0$, where they define a lift of the surface together with the $C_3$-action. These lifts must be of Type XI from \cite[Table 8.9]{CAG}.

Similarly, the equations of Types $(1) (a) (i)$ and $(2) (a) (i)$ define lifts to characteristic $0$ together with the $C_2$-action. Hence, these lifts are of Type XII from \cite[Table 8.9]{CAG}. 

We summarize the classification of automorphism groups of del Pezzo surfaces of degree $2$ in Table \ref{tbl:autodp2} in the Appendix. There, in the first column, we give the name of the corresponding family, both in the notation of Theorem \ref{thm: main2} and in the notation of \cite[Table 8.9]{CAG}. The second and third columns give the group $\Aut(X)$ and its size. In the remaining columns, we list the number of elements of a given Carter conjugacy class in $\Aut(X)$.

\section{Del Pezzo surfaces of degree $1$} \label{sec: deg1}
\subsection{The anti-bicanonical map}
As in the case of degree $2$, we start by describing the geometry of del Pezzo surfaces of degree $d = 1$ and we refer to \cite{Demazure} for characteristic-free facts on del Pezzo surfaces. Recall that the anti-bicanonical system $|-2K_X|$ defines a finite morphism $f: X \to Q$ onto a quadratic cone $Q \subseteq \mathbb{P}^3$. As in degree $2$, it turns out that this map is always separable, even in characteristic $2$.

\begin{proposition} \label{prop: separable}
The anti-bicanonical linear system $|-2K_X|$ defines a finite separable morphism $f: X \to Q$ of degree $2$.
\end{proposition}

\begin{proof}
If $f$ is not separable, then $p = 2$ and $f$ is purely inseparable. 
But then $f$ is a homeomorphism in the \'etale topology. This is impossible, since 
$H_{\et}^2(X,\mathbb{Z}_{\ell})$ has rank $9$ (because $X$ is the blow-up of 
$8$ points in the plane), 
while $H_{\et}^2(Q,\mathbb{Z}_{\ell})$ has rank $1$.
\end{proof}

Let 
$$R(X,-K_X) = \bigoplus_{n=0}^\infty H^0(X,\calO_X(-nK_X))$$
be the graded anti-canonical ring of $X$. By the Riemann--Roch-Theorem, we have 
\begin{eqnarray*}
\dim_\Bbbk R(X,-K_X)_1 &=& 2 \\
\dim_\Bbbk R(X,-K_X)_2 &=& 4 \\
\dim_\Bbbk R(X,-K_X)_3 &=& 7.
\end{eqnarray*}
Thus, we can choose  $u,v$ from $R(X,-K_X)_1$,  $x\in R(X,-K_X)_2 \setminus
S^2(R(X,-K_X)_1)$, and 
 $y\in R(X,-K_X)_3  \setminus S^3(R(X,-K_X)_1)+R(X,-K_X)_1\otimes R(X,-K_X)_2$ and obtain the following relation between the generators
\beq\label{equation}
y^2 + y(a_1x+a_3) + x^3+a_2x^2+a_4x+a_6 = 0,
\eeq
where $a_k$ denotes a binary form of degree $k$ in  $u$ and  $v$.
In particular, via Equation \eqref{equation}, we can view $X$ as a surface of degree $6$ in 
the weighted projective space $\mathbb{P}(1,1,2,3)$, the anti-canonical map is the projection of this surface onto the $u,v$-coordinates, and the anti-bicanonical map is the projection onto the $u,v,x$-coordinates.

If $p\ne 2$, we can replace $y$ with $y+\frac{1}{2}(a_1x+a_3)$ to assume that $a_1 = a_3 = 0$. 
The surface $X$ is a double cover 
of a quadratic cone $Q\cong \bbP(1,1,2)$. The branch curve  $B = V(x^3+a_2x^2+a_4x+a_6)$ is a 
curve of degree $6$ not passing through the vertex of $Q$.  It is a smooth curve of genus $4$ with a vanishing 
theta characteristic $g_3^1$ defined by the ruling of $Q$. If we blow up the vertex of $Q$, we obtain a surface 
isomorphic to 
the rational minimal ruled surface $\mathbf{F}_2$. The preimage of the curve $B$ is a curve in the linear system 
$|6\mathfrak{f}+3\mathfrak{e}|$, where $\mathfrak{f}$ and $\mathfrak{e}$ are the
standard generators of $\Pic(\mathbf{F}_2)$ with 
$\mathfrak{f}^2 = 0$ and $\mathfrak{e}^2 = -2$. The curve $B$ is its canonical model in $\bbP^3$.

In our case, when the characteristic $p =2$, the analog of $B$ is the curve $V(a_1x+a_3)$ in $Q$. In particular, Proposition \ref{prop: separable} tells us that $a_1x + a_3 \neq 0$ and there is no way of removing these terms. Moreover, the curve $B$ always passes through the vertex of $Q$ and its strict transform on $\mathbf{F}_2$ is in $|3\mathfrak{f}|$ if $a_1 = 0$ and in $|3\mathfrak{f} + \mathfrak{e}|$ if $a_1 \neq 0$.
The analogue of the involution $y \mapsto -y$, classically called the Bertini involution, is the involution $\beta$ defined by replacing $y$ with $y+a_1x+a_3$. As in the classical case, we call this $\beta$ {\em Bertini involution}. 

By calculating the partial derivatives in Equation \eqref{equation}, the smoothness of $X$ yields the following restrictions on the $a_i$:

\begin{proposition}\label{prop: smoothness}
In \eqref{equation}, the smoothness of $X$ is equivalent to the condition that the equations
\begin{eqnarray*}
a_1x + a_3 &=& 0 \\
x^2 + a_1y + a_4 &=& 0 \\
a_{1,u}xy + a_{3,u}y + a_{2,u}x^2 + a_{4,u}x + a_{6,u} &=& 0 \\ a_{1,v}xy + a_{3,v}y + 
a_{2,v}x^2 + a_{4,v}x + a_{6,v} &=& 0
\end{eqnarray*}
with $a_{i,u} \coloneqq \frac{\partial a_i}{\partial u}$ and $a_{i,v} \coloneqq \frac{\partial a_i}{\partial v}$ have no common solutions on $X$.
\end{proposition}


\subsection{Normal forms} \label{sec: normalforms}
In this section, we find normal forms for del Pezzo surfaces of degree $1$ in characteristic $2$. In total, we will have $14$ different normal forms, corresponding to the $14$ possible combinations of singularities of 
the ramification curve $R$ and the branch curve $B$. 
First, we simplify the equations of the branch curve.
\begin{lemma} \label{lem: branchcurve}
Let $X$ be a del Pezzo surface of degree $1$ given by Equation \eqref{equation}. Then, after a suitable change of coordinates, we may assume that the equation $a_1x + a_3$ of $B$ is one of the following:
\begin{enumerate}
    \item  $ux + v^3$
    \item  $ux$
    \item $uv(u+v)$
    \item $u^2v$
    \item $u^3$
\end{enumerate}
\end{lemma}
\begin{proof}
If $a_1 \neq 0$, we may assume that $a_1 = u$ after applying a linear substitution in $u$ and $v$. Then, a substitution of the form $x \mapsto x + b_2$ for a suitable binary form $b_2$ of degree $2$ in $ u$ and  $v$ allows us to set $a_3 = \lambda  v^3$. Then, rescaling  $v$, we can assume $\lambda \in \{0,1\}$.

If $a_1 = 0$, we get three cases according to the number of distinct roots of $a_3$. The equation can be normalized by applying a linear substitution in  $u$ and  $v$ to get Cases (3), (4), and (5).
\end{proof}

If we consider $\mathbb{P}(1,1,2)$ as a quadratic cone $Q$ in $\mathbb{P}^3$, these $5$ normal forms for $a_1x + a_3$ correspond to the cases where $B$ is a twisted cubic, a union of a line and a conic, a union of three lines, a union of a double line and a simple line, or a triple line, respectively.
Later, we will use automorphisms of $\mathbb{P}(1,1,2)$ that preserve the equation of $B$ and the form of Equation \eqref{equation} in order to move the images of the singular points of $R$ to special positions. In the following lemma, we describe this group of automorphisms.
\begin{lemma} \label{lem: stabilizerofB}
Let $H \subseteq \Aut(\Bbbk[u,v,x]) \subseteq \Aut(\Bbbk[u,v,x,y])$ be the subgroup of automorphisms that preserve $a_1x + a_3$, act on $x$ as $x \mapsto x + b_2$ for some binary quadratic form $b_2$ in  $u$ and  $v$, and that map Equation \eqref{equation} to one of the same form, with possibly different $a_2,a_4,$ and $a_6$. Then, $H$ consists of substitutions of the 
form
\begin{eqnarray*}
u & \mapsto & \alpha u + \beta v  \\
v & \mapsto & \gamma u + \delta v  \\
x & \mapsto & x + b_2 \\
\end{eqnarray*}
where $\alpha,\beta,\gamma,\delta \in \Bbbk$ such that $\alpha \delta + \beta \gamma \neq 0$, and
\begin{enumerate}
    \item if $a_1x + a_3 = ux +  v^3$, then $\alpha = 1$, $\beta = 0$, $\delta^3 = 1$, 
    $b_2 = \gamma^3 u^2 + \gamma^2\delta uv  + \gamma \delta^2  v^2$. \\
    In particular, $H \cong \Bbbk^+ : 3$.
    \item if $a_1x + a_3 = ux$, then $\alpha = 1$, $\beta = b_2 = 0$. \\
    In particular, $H \cong \Bbbk^+: \Bbbk^{\times}$.
    \item if $a_1x + a_3 = uv(u+v)$, then $\alpha \gamma (\alpha + \gamma) = \beta \delta (\beta+\delta) = 0$, $\alpha^2 \delta + \beta \gamma^2 = \alpha \delta^2 +\beta^2 \gamma = 1$. \\
    In particular, $H \cong \Bbbk[u,v]_2 : (3 \times \mathfrak{S}_3)$
    \item if $a_1x + a_3 =  u^2v$, then $\beta = \gamma = 0$, $\delta = \alpha^{-2}$. \\
    In particular, $H \cong \Bbbk[u,v]_2 : \Bbbk^{\times}$.
    \item if $a_1x + a_3 =  u^3$, then $\beta = 0$, $\alpha^3 = 1$. \\
    In particular, $H \cong \Bbbk[u,v]_2 : (\Bbbk : \Bbbk^{\times} \times 3)$.
\end{enumerate}
\end{lemma}

For the convenience of the reader, we record the effect of a general substitution on the remaining $a_i$ in Equation \eqref{equation}. The proof is a straightforward calculation.

\begin{lemma} \label{lem: coefficientchange}
A substitution of the form 
\begin{eqnarray*}
u & \mapsto & \alpha u + \beta v  \\
v & \mapsto & \gamma u + \delta v  \\
x & \mapsto & x + b_2 \\
y & \mapsto & y + b_1x + b_3
\end{eqnarray*}
where $\alpha,\beta,\gamma,\delta \in \Bbbk$ and $b_i \in \Bbbk[u,v]_i$ such that $\alpha \delta + \beta \gamma \neq 0$, changes the coefficients $(a_2,a_4,a_6)$ in Equation \eqref{equation} as follows
\begin{eqnarray*}
a_2 &\mapsto & \sigma^* a_2 + \sigma^* a_1 b_1 + b_1^2 + b_2 \\
a_4 &\mapsto & \sigma^* a_4 + \sigma^* a_3 b_1 + \sigma^* a_1 b_1 b_2 + \sigma^* a_1b_3 + b_2^2 \\
a_6 &\mapsto & \sigma^* a_6 + \sigma^* a_4b_2 + \sigma^* a_3b_3 + \sigma^* a_2b_2^2 + \sigma^* a_1b_2b_3 + b_3^2 + b_2^3,
\end{eqnarray*}
where $\sigma^* a_i := a_i(\alpha u + \beta v,  \gamma u + \delta v)$.
\end{lemma}

Now, we are ready to describe the normal forms for del Pezzo surfaces of degree $1$.

\begin{theorem}\label{thm: canform} Every del Pezzo surface of degree $1$ in characteristic $2$ 
is a  surface of degree $6$ in $\mathbb{P}(1,1,2,3)$ given by an equation of the form
\beq\label{equation2}
y^2 + y(a_1(u,v)x+a_3(u,v)) + x^3+ a_2(u,v)x^2 + a_4(u,v)x+a_6(u,v) = 0,
\eeq
where $(a_1,a_2,a_3,a_4,a_6)$ is one of the following: 
\pagebreak
 $$
  \resizebox{\textwidth}{!}{
  $
    \begin{array}{|c|c|c|c|c|c|} \hline
    \text{Name}  & a_1x + a_3 & a_2 & a_4 & a_6 & \# \text{Parameters}  \\ \hline \hline
    (1) (a) & ux +  v^3 & a v^2 & b u^4 + c u^2 v^2 + d v^4 & e u^6 + f u^4 v^2 + g u^2 v^4 + h v^6 & 8 \\
    (1) (b) & ux +  v^3 & a v^2 & c u^2 v^2 + d v^4 & e u^6 + f u^4 v^2 + g u^2 v^4 + h v^6 & 7 \\
    (1) (c) & ux +  v^3 & a v^2 & d v^4 & e u^6 + f u^4 v^2 + g u^2 v^4 + h v^6 & 6 \\
    (1) (d) & ux +  v^3 & a v^2 & c u^2 v^2 & e u^6 + f u^4 v^2 + g u^2 v^4 + h v^6 & 6 \\
    (1) (e) & ux +  v^3 & a v^2 & 0 & e u^6 + f u^4 v^2 + g u^2 v^4 + h v^6 & 5 \\
    (2) (a) & ux& a v^2 &  v^4 & b u^6 + d u^4 v^2 + e u^3 v^3 + f u^2 v^4 + g u v^5 + h v^6 & 7 \\
    (2) (b) & ux& a v^2 &  v^4 & b u^6 + d u^4 v^2 + f u^2 v^4 + g u v^5 + h v^6  & 6 \\  (2) (c) & ux& a v^2 &  v^4 & b u^6 + d u^4 v^2 + e u^3 v^3 + f u^2 v^4 + h v^6  & 6 \\
    (2) (d) & ux& a v^2 &  v^4 & c u^5v+ d u^4 v^2  + f u^2 v^4 + h v^6 & 5 \\
    (2) (e) & ux& a v^2 & 0 & b u^6 + d u^4 v^2 + e u^3 v^3 + f u^2 v^4 + eu v^5 + h v^6 & 6 \\
    (2) (f) & ux& a v^2 & 0 & b u^6 + d u^4 v^2 + f u^2 v^4 + uv^5 + h v^6 & 5 \\
    (3) \phantom{(f)} & uv(u+v) & auv & b u^3v + (b+c) u^2 v^2 + cu v^3 & d u^5v+ e u^3 v^3 + fu v^5 & 6 \\
    (4) \phantom{(f)} &  u^2v & 0 & a u^3v + b u^2 v^2 + cuv^3 & d u^5v+ e u^3 v^3 + uv^5 & 5 \\
    (5) \phantom{(f)} &  u^3 & 0 & a u^3v + b u^2 v^2 + cuv^3 & uv^5 + d v^6 & 4 \\ \hline 
    \end{array}
    $
    }
   $$
Moreover, the parameters satisfy the conditions summarized in the following table, where $$\Delta \coloneqq a_3^4 + a_1^3a_3^3 + a_1^4(a_4^2 + a_1a_3a_4 + a_2a_3^2 + a_1^2a_6).$$ In this table, we also describe the singularities of the irreducible components of the reduction $R_{\red}$ of the ramification curve $R$.
 $$
  \resizebox{\textwidth}{!}{
  $
    \begin{array}{|c|c|c|} \hline
    \text{Name}  & \text{Conditions on the parameters} & \text{Singularities of the irreducible components of } R_{\red}  \\ \hline \hline 
   \multirow{2}{*}{$ (1) (a) $}& \Delta \text{ has only simple roots} & \multirow{2}{*}{$ A_2 \text{ over } [1:v:v^3] \text{ with }v^8 + dv^6 + cv^4 + bv^2 = 0$} \\
        &  v^8 + dv^6 + cv^4 + bv^2 \text{ has four distinct roots} &  \\ \hline 
   \multirow{2}{*}{$ (1) (b) $} & \multirow{2}{*}{$\Delta \text{ has only simple roots}, c,d \neq 0 $}  & A_4 \text{ over } [1:0:0] \\
    & &  2 A_2 \text{ over } [1:v:v^3] \text{ with } v^4 + dv^2 + c = 0 \\ \hline 
   \multirow{2}{*}{$ (1) (c) $}& \multirow{2}{*}{$\Delta \text{ has only simple roots}, d \neq 0 $} &  A_6 \text{ over } [1:0:0]\\
    & & A_2 \text{ over } [1:d^{\frac{1}{2}}:d^{\frac{3}{2}}] \\ \hline 
    (1) (d) &   \Delta \text{ has only simple roots}, c \neq 0 & 2A_4 \text{ over } [1:0:0]  \text{ and } [1:c^{\frac{1}{4}}:c^{\frac{3}{4}}]\\ \hline
    (1) (e) & e \neq 0 & A_8 \text{ over } [1:0:0] \\ \hline 
    (2) (a) & u^{-4} \Delta \text{ has only simple roots},e,g,(g^2 + a + h)\neq 0  &  3A_2 \text{ over } [0:1:1],[1:0:0]  \text{ and } [g^{\frac{1}{2}}:e^{\frac{1}{2}}:0]  \\ \hline 
    \multirow{2}{*}{$ (2) (b) $}& \multirow{2}{*}{$ b,g,(g^2 + a + h) \neq 0 $} & A_4 \text{ over } [1:0:0] \\
    && A_2 \text{ over } [0:1:1] \\ \hline 
    (2) (c) & b,e,(a+h) \neq 0 &  3 A_2 \text{ over } [0:1:1],[1:0:0] \text{ and } [0:1:0] \\ \hline 
    \multirow{2}{*}{$(2) (d) $}& \multirow{2}{*}{$ c,(a+h) \neq 0 $} & A_4 \text{ over } [0:1:0] \\
     & & A_2 \text{ over } [0:1:1] \\ \hline 
    (2) (e) & u^{-6} \Delta \text{ has only simple roots}, e \neq 0 & 3A_2 \text{ over } [0:1:0],[1:0:0] \text{ and } [1:1:0] \\ \hline 
    \multirow{2}{*}{$(2) (f) $} & \multirow{2}{*}{$u^{-6} \Delta \text{ has only simple roots} $}& A_4 \text{ over } [1:0:0] \\
    && A_2 \text{ over } [0:1:0] \\ \hline 
    (3) \phantom{(f)} & d,f \neq 0, (d+e+f) \not \in \{0,1\} & 3A_2 \text{ over } [1:0:0],[0:1:0] \text{ and } [1:1:0] \\ \hline 
    (4) \phantom{(f)}& d \neq 0 & 2A_2 \text{ over } [1.0:0] \text{ and } [0:1:0] \\ \hline 
    (5)\phantom{(f)} & - & A_2 \text{ over } [0:1:0] \\ \hline  
    \end{array}
    $
    }
   $$
\end{theorem}

\smallskip

\begin{remark}
The conditions on the parameters that guarantee the smoothness of $X$ are equivalent to the conditions that Equation \eqref{equation2} is the 
Weierstrass equation of an elliptic fibration with only irreducible fibers. We will study this fibration later in Section \ref{sec: strangefibration}. There, the reader can also find an explicit formula for the homogeneous polynomial $\Delta$, which is the discriminant of this fibration.
\end{remark}

\begin{proof}[Proof of Theorem \ref{thm: canform}] By Lemma \ref{lem: branchcurve}, there are, up to choice of coordinates, $5$ possible equations for $B$. We will now give normal forms in each case. 

\smallskip
\noindent \underline{(1) $a_1x + a_3 = ux + v^3$}
\smallskip

\noindent Here, the ramification curve $R$ is given by the two equations
\begin{eqnarray*}
ux + v^3 &=& 0 \\
y^2 + x^3 + a_2x^2 + a_4x + a_6 &=& 0.
\end{eqnarray*}
One checks that the curve $R$ is smooth at the points with $u = 0$. On the affine chart $u = 1$, it is given in $\mathbb{A}^2$ by the single equation
$$
y^2 + v^9 + a_2(1,v)v^6 + a_4(1,v)v^3 + a_6(1,v),
$$
so it has singularities over the roots of the derivative $F'$ of $F := v^9 + a_2(1,v)v^6 + a_4(1,v)v^3 + a_6(1,v)$. After applying an element of $H$ in Lemma \ref{lem: stabilizerofB}, we may assume that $0$ is the root of highest multiplicity of $F'$.

Now, substitutions as in Lemma \ref{lem: coefficientchange} that fix $u,v,$ and $x$ do not change the location of the points that lie under singularities of $R$ and thus, by Lemma \ref{lem: coefficientchange}, we can assume that $a_2 = av^2$, $a_4 = bu^4 + cu^2v^2 + dv^4$, $a_6 = eu^6 + fu^4v^2 + gu^2v^4 + hv^6$.
With this notation, the polynomial $F'$ becomes $v^8 + dv^6 + cv^4 + bv^2$ and the conditions of Proposition \ref{prop: smoothness} boil down to $v^8 + dv^6 + cv^4 + bv^2$ and $\Delta(1,v) = v^{12} + v^9 + (d^2 + a)v^8 + dv^7 + hv^6 + cv^5 + (c^2 + g)v^4 + bv^3 + fv^2 + b^2 + e$ not having a common solution. The former is the derivative of the latter, hence we want that the latter has only simple zeroes. 

Now, if $F'$ has four distinct roots, we are in Case (a). If $F'$ has less than four distinct roots, we may assume $b = 0$. If $F'$ has exactly three roots, then we are in Case (b). If $b = 0$, the polynomial $F'$ has exactly two roots if and only if either $c = 0$ and $d \neq 0$, which is Case (c), or $d = 0$ and $c \neq 0$, which is Case (d). Finally, $F'$ has a single root if and only if $b = c = d = 0$, which is Case (e). 

\smallskip
\noindent  \underline{(2) $a_1x + a_3 = ux$}
\smallskip

\noindent  Here, the ramification curve has two components $R_1$ and $R_2$. The curve $R_1$ is given by
\begin{eqnarray*}
u &=& 0 \\
y^2 + x^3 + a_2x^2 + a_4x + a_6 &=& 0.
\end{eqnarray*}
This curve has a unique singularity, which is of type $A_2$ and located over $[0:1:a_4(0,1)^{\frac{1}{2}}]$. Rescaling $v$, we may assume that $a_4(0,1) \in \{0,1\}$.

The curve $R_2$ is given by
\begin{eqnarray*}
x &=& 0 \\
y^2 + a_6 &=& 0.
\end{eqnarray*}
This curve has singularities over the points $[u:v:0]$ where the derivatives of $a_6$ by $u$ and $v$ both vanish.

First, assume that $a_4(0,1) = 1$ and one of the singularities of $R_2$ does not lie over $[0:1:0]$. Then, using a substitution in $v$ as in Lemma \ref{lem: stabilizerofB}, we can assume that one of them lies over $[1:0:0]$. Substitutions as in Lemma \ref{lem: coefficientchange} which fix $u,v,$ and $x$ do not change the location of these points and, after applying one of them, we may assume that $a_2 = av^2$, $a_4 = v^4$, and $a_6 = bu^6 + du^4v^2 + eu^3v^3 + fu^2v^4 + guv^5 + hv^6$. If $e,g \neq 0$, this is Case (a), if $e = 0$ and $g \neq 0$, this is Case (b), and if $e \neq 0$ and $g = 0$, this is Case (c). The conditions of Proposition \ref{prop: smoothness} boil down to $\Delta(1,v) = v^8 + hv^6 + gv^5 + fv^4 + ev^3 + dv^2 + b$ having only simple roots and $g^2 \neq a + h$. In particular, $(e,g) \neq (0,0)$.

If $a_4(0,1) = 1$, $R_2$ has a unique singularity, and this singularity lies over $[0:1:0]$, then the only odd monomial in $a_6$ is $u^5v$. A substitution of the form $v \mapsto v + \mu u$ and substitutions as in the previous paragraph allow us to assume that $a_2 = av^2$, $a_4 = v^4$, and $a_6 = cu^5v+ du^4v^2 + fu^2v^4 + hv^6$. The conditions of Proposition \ref{prop: smoothness} become $a + h \neq 0$ and $c \neq 0$. This is Case (d).

If $a_4(0,1) = 0$, then Proposition \ref{prop: smoothness} implies that $R_2$ is smooth over $[0:1:0]$. Hence, we can assume that one of the singularities of $R_2$ lies over $[1:0:0]$. Using a substitution as in Lemma \ref{lem: coefficientchange} which fixes $u,v$ and $x$, we may assume that $a_2 = av^2$, $a_4 = 0$, and $a_6 = bu^6 + du^4v^2 + eu^3v^3 + fu^2v^4 + guv^5 + hv^6$. Since $R_2$ is smooth over $[0:1:0]$, we have $g \neq 0$.
If $e \neq 0$, we can scale $v$ so that $g = e$. This is Case (e). If $e = 0$, we scale $v$ so that $g = 1$. This is Case (f). 

\smallskip
\noindent  \underline{(3) $a_1x + a_3 = uv(u+v)$}
\smallskip

\noindent  The curve $B$ has the three irreducible components $B_1,B_2,$ and $B_3$, given by $V(u),V(v),$ and $V(u+v)$, respectively. The corresponding components $R_1,R_2$, and $R_3$ of $R$ are given by
\begin{eqnarray*}
y^2 + x^3 + a_2(0,v)x^2 + a_4(0,v)x + a_6(0,v), & \\
y^2 + x^3 + a_2(u,0)x^2 + a_4(u,0)x + a_6(u,0), & \text{ and } \\
y^2 + x^3 + a_2(u,u)x^2 + a_4(u,u)x + a_6(u,u), &
\end{eqnarray*}
respectively. The singular points of $R_1,R_2$, and $R_3$ lie over $[0:1:a_4(0,1)^{\frac{1}{2}}]$, $[1:0:a_4(1,0)^{\frac{1}{2}}]$, and $[1:1:a_4(1,1)^{\frac{1}{2}}]$, respectively.

A substitution as in Lemma \ref{lem: stabilizerofB} which fixes $u$ and $v$ allows us to set $a_4(0,1) = a_4(1,0) = a_4(1,1) = 0$, that is, that $a_4 = bu^3v+ (b+c)u^2v^2 + cuv^3$ for some $b,c \in \Bbbk$. 
Then, a substitution as in Lemma \ref{lem: coefficientchange} which fixes $u,v,$ and $x$ allows us to set $a_2 = auv$ and $a_6 = du^5v+ eu^3v^3 + fuv^5$. The conditions of Proposition \ref{prop: smoothness} become $d \neq 0, f \neq 0$ and $d+e+f \not \in \{0,1\}$.

\smallskip
\noindent  \underline{(4) $a_1x + a_3 = u^2v$}
\smallskip

\noindent  The curve $B$ has two irreducible components $B_1$ and $B_2$, given by $V(u)$ and $V(v),$ respectively. The corresponding components $R_1$ and $R_2$ of $R$ are given by
\begin{eqnarray*}
y^2 + x^3 + a_2(0,v)x^2 + a_4(0,v)x + a_6(0,v), & \text{ and }\\
y^2 + x^3 + a_2(u,0)x^2 + a_4(u,0)x + a_6(u,0), &
\end{eqnarray*}
respectively. The singular points of $R_1$ and $R_2$ lie over $[0:1:a_4(0,1)^{\frac{1}{2}}]$ and $[1:0:a_4(1,0)^{\frac{1}{2}}]$, respectively.

A substitution as in Lemma \ref{lem: coefficientchange} which fixes $u$ and $v$ allows us to set $a_4(0,1) = a_4(1,0)$ and that $a_2$ is a square. Then, a substitution with $b_2 = b_3 = 0$ allows us to eliminate $a_2$. Finally, a substitution with $b_1 = b_2 = 0$ allows us to assume that $a_6$ contains no squares. If we write $a_6 = du^5v+ eu^3v^3 + fuv^5$, then the 
conditions of Proposition \ref{prop: smoothness} becomes $d \neq 0$ and $f \neq 0$, and we can rescale $f$ to $1$.

\smallskip
\noindent  \underline{(5) $a_1x + a_3 = u^3$}
\smallskip

\noindent  The curve $R$ is given by
$$
y^2 + x^3 + a_2(0,v)x^2 + a_4(0,v)x + a_6(0,v)
$$
and it is singular over $[0:1:a_4(0,1)^{\frac{1}{2}}]$.

We apply the same substitutions as in the previous case to remove $a_2$. 
Then, we apply a substitution as in Lemma \ref{lem: coefficientchange} with $b_2 = b_1^2$
to remove the $v^4$-term in $a_4$.
Next, using a substitution that fixes $u,v,$ and $x$ with $b_1 = 0$, we eliminate the squares in $a_6$, write $a_6 = du^5v+ eu^3v^3 + fuv^5$, and rescale $f$ to $1$. 
After that, a substitution of the form $v \mapsto v + \lambda u$, and eliminating the square again, allows us to set $d = 0$.
Next, a substitution as in Lemma \ref{lem: coefficientchange} which fixes $u$ and $v$, with $b_1 = \lambda u$, $b_2 = \lambda^2 u^2$, and $b_3 = \mu u^3$ for suitable $\lambda$ and $\mu$ allows us to eliminate the $u^4$-term in $a_4$ without changing $a_6$. Finally, we apply a substitution with $b_3 = ev^3$ and rename the parameters to assume that $a_6 = uv^5 + dv^6$. The conditions of Proposition \ref{prop: smoothness} are fulfilled for every choice of parameters. 
\end{proof}

\subsection{Fake tritangent planes and odd theta characteristics} \label{sec: fakebitangent}
It is known that a del Pezzo surface $X$ of degree $1$ contains $240$ $(-1)$-curves (see \cite[8.7]{CAG}, 
where the proof is characteristic free).   
They come in pairs $E_i+E_i' \in |-2K_X|$ with $E_i\cdot E_i' = 3$. The 
Bertini involution $\beta$ swaps the two curves in a pair.
The image of $E_i+E_i'$ under the anti-bicanonical map $f$ is a plane section of $Q$ not passing through the vertex. 

If $p \neq 2$, each of the resulting $120$ planes is a tritangent plane to the branch sextic curve and, 
conversely, every tritangent plane to the branch sextic gives rise to a pair of $(-1)$-curves $E_i+E_i'$ with 
$E_i+E_i'\in |-2K_X|$. 
A tritangent plane intersects the branch curve in twice a positive divisor of degree $3$. This divisor is an odd theta characteristic of the curve. It is known that the number of odd theta characteristics on a smooth 
curve of genus $4$ is equal to $120$.

For arbitrary $p$, we still have the following.

\begin{lemma}\label{lm:4.1} The preimage $f^{-1}(C)$ of an integral conic $C = V(x+b_2)$ 
is a sum of two $(-1)$-curves if and only if it is reducible.
\end{lemma}

\begin{proof} Since $f$ has degree $2$ and $C$ is integral, the curve $f^{-1}(C)$ is reducible 
if and only if it has two irreducible components $L_1$ and $L_2$. These components satisfy 
$L_1+L_2\in |-2K_X|$, $L_1\cdot L_2 = 3$, and $L_1^2 = L_2^2$. Via adjunction, this easily implies that 
$L_1$ and $L_2$ are $(-1)$-curves. The converse is clear.
\end{proof}

So, even if $p = 2$, we have $120$ splitting conics and we call the corresponding planes in $\mathbb{P}^3$ \emph{fake tritangent planes} 
in analogy with the situation in the other characteristics. For the rest of this section, we assume $p = 2$. 

Since the anti-bicanonical map is 
\'etale outside the branch curve $V(A)$, the intersection  $E_i\cap E_i'$ lies on the ramification curve 
$R$.   Let 
$\calL = \calO_R(E_i) \cong \calO_R(E_i')$. It is an invertible sheaf on $C$ of degree 2. 
We have 
$$\calL^{\otimes 2}\cong \calO_R(E_i+E_i')\cong \calO_R(-2K_X).$$
Since $B\in |\calO_{\bbP(1,1,2)}(3)|$, we have $R\in |-3K_X|$. By the adjunction formula, we have
$$\omega_R \cong \calO_R(-3K_X+K_X) \cong \calL^{\otimes 2}.$$
As in the case of degree $2$,
invertible sheaves on $R$ that satisfy this property are called invertible theta characteristics.
Let $\Theta(R)$ be the set of isomorphism classes of such invertible theta characteristics on $R$ and let $J(R)$ be the generalized Jacobian of $R$. As in Lemma \ref{lem: generalizedjacobian}, one can prove that $J(R)$ is a product of additive groups.
\begin{lemma}
The generalized Jacobian $J(R)$ of $R$ is isomorphic to $\mathbb{G}_a^4$.
\end{lemma}

Thus, as in degree $2$, finding fake tritangent planes using theta characteristics on $R$ is subtle in characteristic $2$. We refer to Example \ref{ex: thetacharacteristic} for an example in degree $2$ that further illustrates this point and leave it to the reader to find a similar example in degree $1$.

\subsection{Rational elliptic surfaces} \label{sec: strangefibration}
Equation \eqref{equation} can also serve as the Weierstrass equation of the rational elliptic surface 
$\phi :Y\to \bbP^1$ obtained by blowing up the base point $p_0$ of $|-K_X|$.
Since $X$ is a del Pezzo surface, all members of $|-K_X|$ are irreducible, hence so are all fibers of $\phi$.
The discriminant of $\phi$ is
$$\Delta = a_3^4+a_1^3a_3^3+a_1^4(a_4^2+a_1a_3a_4+a_2a_3^2+a_1^2a_6).$$
The singular fibers of $\phi$ lie over the zeroes of $\Delta$. 
Moreover, the Bertini involution, which is given by 
$\beta: y \mapsto y + (a_1x + a_3)$, induces the inversion on the group structure of each fiber. 
In particular, for $[u_0:v_0] \in \mathbb{P}^1$, if $a_1(u_0,v_0)x + a_3(u_0,v_0) = 0$, 
 the corresponding fiber $F$ of $\phi$ is cuspidal, if $a_1(u_0,v_0) = 0$ and $a_3(u_0,v_0) \neq 0$, then $F$ 
is smooth and supersingular, and in the other cases, $F$ is either nodal, or smooth and ordinary, according to whether $\Delta(u_0,v_0)$ is zero or not. Applying these observations to the normal forms of Theorem \ref{thm: canform}, we obtain the following information on $\phi$.

\begin{proposition}\label{prop: strangefibration}
Let $X$ be a del Pezzo surface of degree $1$ given by one of the normal forms in Theorem \ref{thm: canform}. Then, the associated elliptic fibration $\phi$ is elliptic and all its fibers are irreducible. The discriminant $\Delta$ and the singular fibers of $\phi$ are given in the following table.

 $$
  \resizebox{\textwidth}{!}{
  $
    \begin{array}{|c|c|c|c|} \hline
    \text{Name}  & \Delta & \text{Nodal fibers over the} & \text{Cuspidal fibers over} \\ \hline \hline 
   \multirow{2}{*}{$ (1) (a) $}& v^{12} +  u^3 v^9 + (d^2 + a) u^4 v^8 + d u^5 v^7 + h u^6 v^6 + c u^7 v^5  & \multirow{2}{*}{$12 \text{ roots of } \Delta $}& \multirow{2}{*}{--} \\
    &  +  (c^2 + g) u^8 v^4 + b u^9 v^3 + f u^{10} v^2 + (b^2 + e) u^{12}  & & \\ \hline 
\multirow{2}{*}{$ (1) (b) $} & v^{12} +  u^3 v^9 + (d^2 + a) u^4 v^8 + d u^5 v^7 + h u^6 v^6 + c u^7 v^5  & \multirow{2}{*}{$12 \text{ roots of } \Delta$} & \multirow{2}{*}{--} \\
    &  +  (c^2 + g) u^8 v^4  + f u^{10} v^2 + e u^{12}  & & \\ \hline 
\multirow{2}{*}{$(1) (c) $}& v^{12} +  u^3 v^9 + (d^2 + a) u^4 v^8 + d u^5 v^7 + h u^6 v^6  & \multirow{2}{*}{$ 12 \text{ roots of } \Delta $} & \multirow{2}{*}{--} \\
& + g u^8 v^4 + f u^{10} v^2 + e u^{12} & & \\ \hline 
\multirow{2}{*}{$(1) (d) $} & v^{12} +  u^3 v^9 +  a u^4 v^8 + h u^6 v^6 + c u^7 v^5  & \multirow{2}{*}{$12 \text{ roots of } \Delta $} & \multirow{2}{*}{--}\\
&  + (c^2 + g) u^8 v^4 + f u^{10} v^2 + e u^{12} & & \\ \hline 
\multirow{2}{*}{$(1) (e) $}&  v^{12} +  u^3 v^9 + a u^4 v^8  + h u^6 v^6   & \multirow{2}{*}{$ 12 \text{ roots of } \Delta$} & \multirow{2}{*}{--} \\
&  +  g u^8 v^4 + f u^{10} v^2 + e u^{12} & & \\ \hline 
(2) (a) & u^4(v^8 +  u^2(b u^6 + d u^4 v^2 + e u^3 v^3 + f u^2 v^4 + gs v^5 + h v^6)) & 8 \text{ roots of } u^{-4} \Delta & [0:1] \\ \hline 
(2) (b) & u^4(v^8 +  u^2(b u^6 + d u^4 v^2 + f u^2 v^4 + gs v^5 + h v^6)) &  8 \text{ roots of } u^{-4} \Delta &  [0:1] \\ \hline  (2) (c) & u^4(v^8 +  u^2(b u^6 +  d u^4 v^2 + e u^3 v^3 + f u^2 v^4 + h v^6)) & 8 \text{ roots of } u^{-4} \Delta & [0:1] \\ \hline 
(2) (d) & u^4(v^8 +  u^2(c u^5v+ d u^4 v^2 + f u^2 v^4 + h v^6)) &  8 \text{ roots of } u^{-4} \Delta &  [0:1] \\ \hline 
 \multirow{2}{*}{$(2) (e)$} & \multirow{2}{*}{ $u^6(b u^6 + d u^4 v^2 + e u^3 v^3 + f u^2 v^4 + e uv^5 + h v^6) $}& \text{If } h \neq 0: 6 \text{ roots of } u^{-6} \Delta & \multirow{2}{*}{$ [0:1] $ } \\
 && \text{If } h = 0:  5 \text{ roots of } u^{-7} \Delta &   \\ \hline 
 \multirow{2}{*}{$ (2) (f) $} & \multirow{2}{*}{$ u^6(b u^6 + d u^4 v^2 + f u^2 v^4 + uv^5 + h v^6) $} & \text{If } h \neq 0:  6 \text{ roots of } u^{-6} \Delta &  \multirow{2}{*}{$ [0:1] $} \\
 && \text{If } h = 0: 5 \text{ roots of } u^{-7} \Delta &   \\ \hline 
 (3) \phantom{(f)} & u^4v^4(u+v)^4 & $--$ &  [1:0],[0:1],[1:1] \\ \hline 
 (4) \phantom{(f)} & u^8v^4 & $--$ & [1:0],[0:1] \\ \hline 
 (5) \phantom{(f)} & u^{12} & $--$ &  [0:1]  \\ \hline 
    \end{array}
    $
     }
   $$
\end{proposition}

\smallskip
 
\begin{remark}
The classification of singular fibers of rational elliptic surfaces with a section in characteristic $2$ can be found in \cite{Lang}.
Lang shows that in the cases where the general fiber is a supersingular elliptic curve, the
number of singular fibers is at most $3$, which agrees with our computations. In fact, our normal forms also yield normal forms for all rational elliptic surfaces with a section in characteristic $2$ 
whose singular fibers are irreducible. 
\end{remark}

\section{Automorphism groups of del Pezzo surfaces of degree $1$} \label{sec: automorphisms}
This section consists of three parts. In the first part, we collect various restrictions on the group $G(X) = \Aut(X)/\langle \beta \rangle$ arising from the geometry of $X$. In the second part, we give an explicit description of $\Aut(X)$ in terms of Equation \eqref{equation} and use it to classify all surfaces where $G(X)$ is non-trivial and to determine the group $\Aut(X)$ in every case.
In the third part, we compare our classifiction with the classification in characteristic $0$ from \cite[Table 8.14]{CAG} and use this to determine the conjugacy classes of all elements in $\Aut(X)$ (see Table \ref{tbl:autodp1} in the Appendix).
Throughout, we assume $p = 2$.

\subsection{Restrictions on $G(X)$}
Since the elliptic fibration $\phi:Y \to \mathbb{P}^1$ associated to $X$ is obtained by blowing up the base point of $|-K_X|$, we can identify $\Aut(X)$ with the subgroup of $\Aut(Y)$ preserving a chosen section.
Let $r:\Aut(X)\to \Aut(\bbP^1)$ be the natural homomorphism defined by the action of $\Aut(X)$ on the coordinates $[u:v]$ of the 
base of $\phi$. Since $\phi$ is the unique relatively minimal smooth proper model of its generic fiber $F_{\eta}$, the kernel $K = \Ker(r)$ is isomorphic to the group of automorphisms of the elliptic curve $F_{\eta}$. 
In particular, $K$ contains the Bertini involution $\beta$ and it can contain more automorphisms only if the $j$-invariant of 
$F_{\eta}$ is equal to $0 = 1728$, in which case $K$ is a subgroup of $Q_8 : 3 \cong \SL_2(\mathbb{F}_3)$. 

Let $P$ be the image of $r$. Evidently, $P$ is a finite subgroup of $\Aut(\bbP^1)$ that leaves invariant the set $S_1$ of points 
$p = [u_i:v_i]$ corresponding to the singular fibers. 
It also leaves invariant the set $S_2$ of the projections of singular points of the irreducible components of the ramification 
curve $R$.

The following proposition shows what kind of groups can be expected to occur for $P$. We use the 
known classification of  finite subgroups of $\Aut(\bbP^1) \cong \PGL_2(\Bbbk) \cong \SL_2(\Bbbk)$ \cite[Theorem 2.5]{DM}.

\begin{proposition}
The group $P$ is isomorphic to $G_{\xi,A}$ or $D_{2n}$.
\end{proposition}
\begin{proof}
Since $\SL_2(2) \cong \mathfrak{S}_3 \cong D_6$, it suffices to show that $\SL_2(\bbF_q) \not \subseteq P$ 
for $q = 2^m$ and $m \geq 2$.
Since the set $S_2$ has cardinality at most $4$ and $P$ preserves $S_2$, every homogeneous polynomial $F$ with simple roots along $S_2$ is $P$-semi-invariant of degree at most $4$. On the other hand, by \cite[Theorem 6.1.8]{Neusel}, 
the ring $\Bbbk[u,v]^{\SL_2(\bbF_q)}$ is generated over $\bbF_q$ by the 
Dickson polynomials $\mathbf{L}$ and $\mathbf{d}_{2,1}$ of degrees $q+1$ and $q^2-q$,
respectively. If $\SL_2(q) \subseteq P$, then $F$ is also a semi-invariant polynomial for $\SL_2(q)$ and
 if $q \neq 2$, then $\SL_2(\bbF_q)$ is simple, so $F \in \Bbbk[u,v]^{\SL_2(q)} = \Bbbk[\mathbf{L},\mathbf{d}_{2,1}]$. Hence, $q = 2$, as claimed.
\end{proof}

We recall from \cite[\S 1.3]{DM} that the image of the Bertini involution $\beta$ under the injective homomorphism 
$\rho:\Aut(X)\to W(\sfE_8)$ is equal to  $-\id_{\sfE_8}$. However, in contrast to the situation in degree $2$,
the extension  $W(\sfE_8)\to W(\sfE_8)/(-\id_{\sfE_8}) \cong \Or_8^+(2)$ does not split. The semi-direct product 
$W(\sfE_8) = 2.\textrm{GO}_8^+(2)$ corresponds to a non-trivial homomorphism $\Or_8^+(2)\to C_2$, whose
 kernel is a simple group $\Or_8(2)$, where we use the  ATLAS notation.

Therefore, in order to determine $\Aut(X)$, it is not enough to determine the 
image $G(X)$ of the homomorphism $\Aut(X) \to \Aut(X)/\langle \beta \rangle$, and thus the calculation of $\Aut(X)$ 
is more complicated than in the case of del Pezzo surfaces of degree $2$. 

Let us summarize the restrictions on $\Aut(X)$ and $G(X)$ that we have collected by now.

\begin{theorem}
Let $X$ be a del Pezzo surface of degree $1$ in characteristic $2$. Let $G(X)$ be the image of the 
homomorphism $\Aut(X) \to \Aut(\mathbb{P}(1,1,2))$, let $K$ be the kernel of the homomorphism $r: \Aut(X) \to \Aut(\bbP^1)$, let 
$P$ be the image of $r$, and let $\phi: Y \to \bbP^1$ be the elliptic fibration associated to $X$. Then, the following hold:
\pagebreak
\begin{enumerate}
    \item[(i)] $\Aut(X)$ is a central extension of $G(X)$ by $\langle \beta \rangle \cong C_2$.
    \item[(ii)] $\Aut(X)$ is an extension of $P$ by $K$.
    \item[(iii)] $\Aut(X)$ is a subgroup of $W(E_8)$.
    \item[(iv)] $G(X)$ is a subgroup of $\Or_8^+(2)$.
    \item[(v)] $K$ is the automorphism group of the generic fiber of $\phi$.
    \item[(vi)] $P$ is isomorphic to $G_{\xi,A}$ or $D_{2n}$.
    \item[(vii)] $P$ preserves the set $S_1$ of points lying under singular fibers of $\phi$. 
    Moreover, it preserves the decomposition of $S_1$ into subsets corresponding to isomorphic fibers.
    \item[(viii)] $P$ preserves the set $S_2$ of points lying under the singularities of $R$. 
    Moreover, it preserves the decomposition of $S_2$ into subsets of isomorphic singularities.
    \item[(ix)] The $j$-function of $\phi$ is $P$-invariant.
\end{enumerate}
\end{theorem}

This yields the following preliminary restrictions on $\Aut(X)$ and $G(X)$.

\begin{corollary} \label{cor: preliminary}
Let $X$ be a del Pezzo surface of degree $1$ in characteristic $2$ given by one of the normal forms in Theorem \ref{thm: canform}.
\begin{enumerate}
    \item[(i)] In Case (1), $G(X)$ is a subgroup of $A_4$.
    \item[(ii)] In Cases  (2) (a), (2) (b), (2) (c), and (2) (d), $G(X)$ is a subgroup of $2^3$.
    \item[(iii)] In Cases (2) (e) and (2) (f), $G(X)$ is a subgroup of $C_5$ or $C_2$.
    \item[(iv)] In Case (3), $K$ is a subgroup of $\SL_2(3)$ and $P$ is a subgroup of $\mathfrak{S}_3$.
    \item[(v)] In Case (4), $K$ is a subgroup of $\SL_2(3)$ and $P$ is cyclic of order $1,3,5,7,9,$ or $15$.
    \item[(vi)] In Case (5) $K$ is a subgroup of $\SL_2(3)$ and $P \cong G_{\xi,A}$, where $\xi$ is a primitive $n$-th root of unity with $n \in \{1,3,5,7,9,15\}$.
\end{enumerate}
\end{corollary}

\begin{proof}
In Case (1), the generic fiber of $\phi$ is ordinary, hence $K = \langle \beta \rangle$ and $G(X) \cong P$. 
The fibration $\phi$ has $12$ nodal fibers, hence the $j$-function has $12$ poles, 
so $|P| \mid 12$. Since $P$ is isomorphic to $G_{\xi,A}$ or $D_{2n}$ with $n$ odd, this implies that 
$P$ is isomorphic to a subgroup of $A_4$.

In Cases (2), we also have $K = \langle \beta \rangle$ and $G(X) \cong P$. 
In Cases (2) (a), (2) (b), (2) (c), and (2) (d), the fibration $\phi$ has $8$ nodal fibers, hence 
$|P| \mid 8$. This implies that $P$ is elementary abelian of order $1,2,4$ or $8$. 
In Cases (2) (e) and (2) (f), the fibration $\phi$ has $5$ or $6$ nodal fibers. If it has $5$ nodal
 fibers, then $|P| \mid 5$, hence $P$ is a subgroup of $C_5$. 
 If it has $6$ nodal fibers, then $P$ is either a subgroup of $C_2$ or 
 isomorphic to the dihedral group $D_6$. In the latter case, 
 $P$ acts without fixed point on $\bbP^1$, which is impossible, since $\phi$ admits a unique cuspidal fiber.

In Case (3), we have $K \subseteq \SL_2(3)$, since the generic fiber of $\phi$ is supersingular. 
Since $\phi$ has three singular fibers, $P$ is isomorphic to a subgroup of $\mathfrak{S}_3$.

In Case (4) we also have $K \subseteq \SL_2(3)$. Since one of the components of $R$ is 
reduced and the other is not, $P$ acts trivially on $S_2$, hence with two fixed points on 
$\mathbb{P}^1$. So, $P$ is cyclic of odd order. Moreover, $P$ is a subgroup of $\Or_8^+(2)$. 
In particular, $P$ admits a faithful representation of dimension at most $8$. Hence, 
if we denote Euler's totient function by $\varphi$, then $\varphi(|P|) \leq 8$. Thus, $P$ is 
of order $1,3,5,7,9$ or $15$. 

In Case (5) we have $K \subseteq \SL_2(3)$ and the action of $P$ on $\bbP^1$ fixes the point 
lying under the unique singular fiber of $\phi$, hence $P \cong G_{\xi,A}$. The order of 
$\xi$ can be bounded by the same argument as in the previous paragraph.
\end{proof}

In particular, we get upper bounds on the size of $\Aut(X)$ in every case. Further information on the $2$-groups that can occur in Case (5) can be obtained using the following remark.

\begin{remark}
Since the maximal powers of $2$ that divide $|W(E_8)|$ and $|W(D_8)|$ are both $2^{14}$, and since $W(D_8)$ is a subgroup of $W(E_8)$, the $2$-Sylow subgroups $P$ in $W(E_8)$ are isomorphic to the $2$-Sylow subgroups in $W(D_8) = 2^7 : \mathfrak{S}_8$. Hence, $P$ is isomorphic to $2^7 : (\mathfrak{S}_8)_2$, where $2^7$ acts on $\bbZ^8$ by an even number of sign changes and $(\mathfrak{S}_8)_2$ is a $2$-Sylow subgroup of $\mathfrak{S}_8$ acting as permutations on $\bbZ^8$. 
The group $(\mathfrak{S}_8)_2$ is isomorphic to the symmetry group of a binary tree of depth $3$, 
considered as a subgroup of $\mathfrak{S}_8$ via the permutation it induces on the leaves of the tree. An equivalent description is as the wreath product  $D_8\wr C_2$, where $D_8\times D_8$ is a subgroup of 
$\frakS_4\times \frakS_4\subset \frakS_8$. The Bertini involution $\beta$ corresponds to the element $(-1,{\rm id})$ that changes all signs.
The $2$-groups that can occur in Corollary \ref{cor: preliminary} are isomorphic to subgroups of $P$.
\end{remark}

In the following example, we apply this remark to give an explicit description of the group $2_+^{1+6}$, which will occur in our classification.

\begin{example} \label{ex: extraspecial}
With notation as in the previous remark, let $G \subseteq P$ be a subgroup containing $\beta$ such that $G/\langle \beta \rangle$ is an elementary abelian $2$-group and such that $\beta \in Q_8 \subseteq G$. Then, each element of $G$ is of the form $(\sigma,\tau)$, where ${\rm ord}(\tau) \leq 2$ and either $\tau$ preserves the set of coordinates whose sign is changed by $\sigma$ and then $(\sigma,\tau)$ has order $1$ or $2$, or $\tau$ swaps this set with the set of coordinates whose sign is not changed and then $(\sigma,\tau)$ has order $4$. In particular, in the latter case, $\tau$ has cycle type $(2,2,2,2)$.
Since $Q_8 \subseteq G$, the image of $G \to (\mathfrak{S}_8)_2$ contains a subgroup $H$ of order $4$ generated by involutions of cycle type $(2,2,2,2)$. The centralizer $C$ of $H$ is of order $8$ and its non-trivial elements are involutions of cycle type $(2,2,2,2)$. The kernel of $G \to (\mathfrak{S}_8)_2$ consists of sign changes $\sigma$ that are compatible with all $\tau \in H$ in the sense that $(\sigma,\tau)^2 \in \langle \beta \rangle$. One checks that the group $N$ of all such compatible sign changes has order $16$ and that all elements of $N$ are also compatible with $C$. Then, $G$ is a subgroup of the resulting extension $M$ of $C$ by $N$.

We have $M/\langle \beta \rangle = 2^6$. This is a 
quadratic space over $\mathbb{F}_2$ with the quadratic form 
$q: M/\langle \beta \rangle  \to \langle \beta \rangle$ defined as $q(x) = \tilde{x}^2$, 
where $\tilde{x}$ is a lift of $x$ to $M$. The subspace $N/\langle \beta \rangle$ is 
totally isotropic of dimension $3$ and the description of $M$ in the 
previous paragraph shows that $q$ is non-degenerate. 
Hence, by \cite[(23.10)]{Aschbacher}, $M$ is isomorphic to 
the extra-special $2$-group of $2_+^{1+6}$.
\end{example}

\subsection{Classification}

Recall that $X$ is a hypersurface of degree $6$ in $\mathbb{P}(1,1,2,3)$ given by Equation \eqref{equation}.
An automorphism of $\mathbb{P}(1,1,2,3)$ is induced by a substitution of the form
\begin{eqnarray*}
u & \mapsto & \alpha u + \beta v  \\
v & \mapsto & \gamma u + \delta v  \\
x & \mapsto & \varepsilon x + b_2 \\
y & \mapsto & \zeta y + b_1x + b_3
\end{eqnarray*}
where $\alpha,\beta,\gamma,\delta,\varepsilon, \zeta \in \Bbbk$, $b_i \in \Bbbk[u,v]_i$, and $\alpha \delta + \beta \gamma , \varepsilon, \zeta \neq 0$. The substitutions that induce the identity on $\mathbb{P}(1,1,2,3)$ are the ones with $\beta,\gamma,b_1,b_2,b_3 = 0$ and $\gamma = \alpha, \varepsilon = \alpha^2, \zeta = \alpha^3$.

Since $X$ is anti-canonically embedded into $\mathbb{P}(1,1,2,3)$, all automorphisms of $X$ are induced by the substitutions as above that map Equation \eqref{equation} to a multiple of itself. Clearly, we can represent every such automorphism by a substitution with $\zeta = 1$. Then, the substitution does not change the coefficient of $y^2$ in Equation \eqref{equation}, hence $\varepsilon^3 = 1$. Therefore, we may assume $\varepsilon = 1$ as well.
In particular, using Lemma \ref{lem: coefficientchange}, we obtain the following description of $\Aut(X)$, where we write $\sigma$ for the substitution
\begin{eqnarray*}
u & \mapsto & \alpha u + \beta v  \\
v & \mapsto & \gamma u + \delta v  \\
\end{eqnarray*}
and $\sigma^* a_i := a_i(\alpha u + \beta v,  \gamma u + \delta v)$.

\begin{lemma} \label{lem: automorphisms}
Let $X$ be a del Pezzo surface of degree $1$ given by Equation \eqref{equation}. Then, $\Aut(X)$ can be identified with the group of $4$-tuples $(b_1,b_2,b_3,\sigma)$, where $b_i \in \Bbbk[u,v]_i$ and $\sigma \in \GL_2(\Bbbk)$ such that
\begin{eqnarray*}
\sigma^*a_1 + a_1 &=& 0, \\
\sigma^*a_2 + a_2 &=& a_1b_1 + b_1^2 + b_2, \\
\sigma^*a_3 + a_3 &=& a_1b_2, \\
\sigma^*a_4 +  a_4 &=& a_3b_1 + a_1b_3 + b_2^2, \\
\sigma^*a_6 +  a_6 &=& a_4b_2 + + a_3(b_3 + b_1b_2) + a_2b_2^2 + a_1(b_2b_3 + b_1b_2^2) + b_3^2 + b_2^3 + b_1^2b_2^2 
\end{eqnarray*}
and where the composition is given by
$$(b_1,b_2,b_3,\sigma) \circ (b_1',b_2',b_3',\sigma') = (\sigma'^*b_1 + b_1',\sigma'^*b_2 + b_2', \sigma'^*b_3 + b_3' + \sigma'^*b_1 b_2',\sigma \circ \sigma')$$
\end{lemma}

In particular, there is a homomorphism $\Aut(X) \to H \subseteq \Aut(\mathbb{P}(1,1,2))$, where $H$ is the group from Lemma \ref{lem: stabilizerofB}. 

\begin{lemma}
The kernel of the homomorphism $\Aut(X) \to H$ is generated by the Bertini involution.
\end{lemma}
\begin{proof}
Let $(b_1,b_2,b_3,\sigma)$ be in the kernel. Then, $\sigma = {\rm id}$ and $b_2 = 0$. The conditions $\sigma^*a_2 = a_2 + a_1b_1 + b_1^2$, $\sigma^*a_4 = a_4 + a_3 b_1 + a_1 b_3$, and $\sigma^*a_6 = a_6 + a_3b_3 + b_3^2$ show that $(b_1,b_3) \in \{(0,0),(a_1,a_3)\}$, so we recover our explicit description of the Bertini involution.
\end{proof}
 
 Now, we use the normal forms of Theorem \ref{thm: canform} to classify all del Pezzo surfaces $X$ of degree $1$ with non-trivial $G(X)$.

 \newgeometry{,vmargin=3cm,hmargin=5cm} 
\begin{landscape}
\begin{theorem} \label{thm: main1}
Every del Pezzo surface of degree $1$ in characteristic $2$ such that $G(X)$ is non-trivial is a surface of degree $6$ in $\bbP(1,1,2,3)$ given by an equation of the form
$
y^2 + (a_1x+a_3)y + x^3 + a_2x^2 + a_4x + a_6 $
where $(a_1,a_2,a_3,a_4,a_6,G(X),\Aut(X))$ is one of the following:
\hspace{-0.5cm}
$$
 \resizebox{22cm}{!}{
  $
    \begin{array}{|c|c|c|c|c|c|c|c|} \hline
    \text{Name}  & a_1x + a_3 & a_2 & a_4 & a_6 & G(X) & \Aut(X) &\# \text{Parameters}  \\ \hline \hline
    (1) (a) (i)\phantom{iiv} & ux + v^3 & av^2 & bu^4 + (b+1)u^2v^2 & eu^6 + fu^4v^2 + (a + b + b^2 + f)u^2v^4 + bv^6 & 2 & 4 & 4 \\
    (1) (a) (ii)\phantom{iv} & ux + v^3 & 0& bu^4  & eu^6 + hv^6 & 3 & 6 & 3 \\
    (1) (a) (iii)\phantom{v} & ux + v^3 & av^2 & u^4  & eu^6 + au^4v^2 + v^6 & 2^2 & Q_8 & 2 \\
    (1) (a) (iv)\phantom{ii} & ux + v^3 & 0 & u^4  & eu^6 + v^6 & A_4 & \SL_2(3)& 1 \\
    (1) (d) (i)\phantom{iiv} & ux + v^3 & av^2 & u^2v^2 & eu^6 + fu^4v^2 + (a + f)u^2v^4  & 2 & 4 & 3 \\
    (1) (e) (i)\phantom{iiv} & ux + v^3 & 0 & 0  & eu^6 + hv^6 & 3 & 6 & 2 \\
        (2) (a) (i) \phantom{iiv} & ux& av^2 & v^4 & bu^6 + (efg^{-1} + e^{\frac{3}{2}}g^{-\frac{1}{2}} +  e^3g^{-3})u^4v^2 + eu^3v^3 + fu^2v^4 + g uv^5 + e^{-\frac{1}{2}} g^{\frac{3}{2}}v^6 & 2 & 2^2 & 5 \\
    (2) (d) (i) \phantom{iiv}& ux& av^2 & v^4 & cu^5v+ du^4v^2  + fu^2v^4 & 2^3 & 2^4 & 4 \\
    (2) (e) (i) \phantom{iiv} & ux& av^2 & 0 & bu^6 + (e+f)u^4v^2 + eu^3v^3 + fu^2v^4 + euv^5 + ev^6 & 2 & 2^2 & 4 \\
    (2) (f) (i) \phantom{iiv} & ux & 0& 0 & bu^6 + uv^5 & 5 & 10 & 1 \\
    (3) (i)\phantom{iiv}  \phantom{(f)}  &  uv(u+v) & auv & bu^3v+ buv^3 & du^5v+ eu^3v^3 + duv^5 & 2 & 2^2 & 4 \\
    (3) (ii)\phantom{iv} \phantom{(f)} & uv(u+v) & auv & a^{\frac{1}{2}}u^3v+ a^{\frac{1}{2}}uv^3 & (e + e^{\frac{1}{2}})u^5v+ eu^3v^3 + (e + e^{\frac{1}{2}}) uv^5 & \mathfrak{S}_3 &  2 \times \mathfrak{S}_3 & 2 \\
    (3) (iii) \phantom{v} \phantom{(f)} &  uv(u+v) & 0 & 0 & du^5v+ eu^3v^3 + d uv^5 & 6& 2 \times 6 & 2 \\
    (3) (iv) \phantom{ii} \phantom{(f)} &  uv(u+v) & 0 & bu^3v+ \zeta_3b u^2v^2 + \zeta_3^2 b uv^3 & (e + e^{\frac{1}{2}})u^5v+ eu^3v^3 + (e + e^{\frac{1}{2}}) uv^5 & 3& 6 & 2 \\
    (3) (v) \phantom{iii} \phantom{(f)} &  uv(u+v) & 0 & 0 & (e + e^{\frac{1}{2}})u^5v+ eu^3v^3 + (e + e^{\frac{1}{2}}) uv^5 & 3 \times \mathfrak{S}_3 & 6 \times \mathfrak{S}_3 & 1 \\
        (4) (i)\phantom{iiv}  \phantom{(f)} & u^2v& 0 & 0 & du^5v+ eu^3v^3 + uv^5  & 3 & 6 & 2 \\
    (5) (i)\phantom{iiv}  \phantom{(f)} & u^3 & 0 & au^3v+ bu^2v^2 &  uv^5 + dv^6 & 2^6 &2_+^{1+6} & 3 \\
    (5) (ii)\phantom{iv}  \phantom{(f)} & u^3 & 0 & 0&  uv^5 + dv^6 & 2^6 : 3 & 2_+^{1+6} : 3 & 1 \\
    (5) (iii)\phantom{v}  \phantom{(f)} & u^3 & 0 & 0&  uv^5 & 2^6 : 15 &2_+^{1+6} : 15 & 0 \\ \hline
      \end{array}
    $
    }
$$
\smallskip
\noindent  
  
Here, $\mathfrak{S}_3$, $D_8$, $Q_8$, and $2_+^{1+6}$, denote 
the symmetric group on $3$ letters, the dihedral group of order $8$, 
the quaternion group, and the even extra-special group of order $128$, respectively. 
In each case, the parameters have to satisfy the conditions of Theorem \ref{thm: canform} and the obvious genericity conditions that keep them from specializing to other subcases.

\end{theorem}
 \end{landscape}
    \restoregeometry

\begin{proof}
We use the normal forms of Theorem \ref{thm: canform} and let $H$ be the group of Lemma \ref{lem: stabilizerofB}. By Lemma \ref{lem: automorphisms}, we have $G(X) \subseteq H$. We apply Lemma \ref{lem: automorphisms} to calculate $\Aut(X)$.

\smallskip
\noindent \underline{Case (1) (a)}
\smallskip

 \noindent  Let $(b_2,\sigma) \in H$. If $(b_2,\sigma) \in G(X)$, then $\sigma$ permutes the roots of the polynomial $F' := v^8 + dv^6 + cv^4 + bv^2$, since these are determined by the singularities of $R$. We have
 $$
 \sigma^* F' = \delta^2 v^8 + dv^6 + \delta(\gamma^2 d + c)v^4 + \delta^2(\gamma^4 d + b)v^2 + \gamma^8 + \gamma^6 d + \gamma^4 c + \gamma^2 b.
 $$
 If $d \neq 0$, this is a multiple of $F'$ if and only if $\delta = 1$ and $\gamma = 0$, hence $\sigma$ is the identity and $G(X)$ is trivial.
 If $d = 0$, it is a multiple of $F'$ if and only if
 \begin{equation}
 \gamma^8 + \gamma^4 c + \gamma^2 b = 0 \label{eqn: 1a1}
 \end{equation}
 and $\delta = 1$ or $c = 0$.
 
 So, assume first that $c \neq 0$ and $\delta = 1$. If $(b_2,\sigma) \in G(X)$, then there exist polynomials $b_1$ and $b_3$ such that $\sigma^*a_2 = a_2 + a_1b_1 + b_1^2 + b_2$ and $\sigma^*a_4 = a_4 + a_3b_1 + a_1b_3 + b_2^2$. In our case, this means
 \begin{eqnarray*}
 0 &=& \gamma^2 au^2 + ub_1 + b_1^2 + \gamma^3 u^2 + \gamma^2 uv + \gamma v^2 \\
 0 &=& \gamma^2 c u^4 + v^3b_1 + ub_3 + \gamma^6 u^4 + \gamma^4 u^2v^2 + \gamma^2 v^4,
 \end{eqnarray*}
 hence $b_1 = \lambda u + \gamma^2 t$ with $\lambda^2 + \lambda = \gamma^2 a + \gamma^3$ and $\gamma^4 = \gamma$, and $b_3 = (\gamma^2 c + \gamma^3)u^3 + \gamma uv^2 + \lambda v^3$. If $\gamma \neq 0$, then $\gamma^4 = \gamma$ implies $\gamma^3 = 1$. Modifying the equation of $X$ by an element of $H$, we may assume that $\gamma = 1$. Plugging this into Equation \eqref{eqn: 1a1}, we obtain $c = b+1$. Hence, $b_1 = \lambda u + v$ with $\lambda^2 + \lambda = a$ and $b_3 = bu^3 + uv^2 + v^3$. 
 Plugging this into the equation for $\sigma^* a_6$ and comparing coefficients in Lemma \ref{lem: automorphisms}, we obtain the conditions $h = b$ and $g = a + b + b^2 + f$. 
 Since $\gamma$ is uniquely determined by Equation \eqref{eqn: 1a1}, we have $G(X) \cong C_2$. The square of any lift of a non-trivial element of $G(X)$ to $\Aut(X)$ is the Bertini involution, hence $\Aut(X) \cong C_4$.
 
 Next, assume that $c = 0$. If $(b_2,\sigma) \in G(X)$, then there exist polynomials $b_1$ and $b_3$ such that $\sigma^*a_2 = a_2 + a_1b_1 + b_1^2 + b_2$ and $\sigma^*a_4 = a_4 + a_3b_1 + a_1b_3 + b_2^2$. In our case, this means
 \begin{eqnarray*}
 0 &=& \gamma^2 au^2 + (1 + \delta^2) av^2 + ub_1 + b_1^2 + \gamma^3 u^2 + \gamma^2 \delta uv + \gamma \delta^2 v^2 \\
 0 &=& v^3b_1 + ub_3 + \gamma^6u^4 + \gamma^4 \delta^2 u^2v^2 + \gamma^2 \delta^4 v^4,
 \end{eqnarray*}
 hence $b_1 = \lambda u + \gamma^2 \delta v$ with $\lambda^2 + \lambda = \gamma^2 a + \gamma^3$ and $\gamma^4 + \gamma = (1 + \delta)a$, as well as $b_3 = \gamma^6 u^3 + \gamma^4 \delta^2 uv^2 + \lambda v^3$.
 
First, assume that $\delta \neq 1$. Then, $\sigma$ has order $3$, hence if $(b_2,\sigma) \in G(X)$, then it fixes one of the four roots of $F'$. After conjugating by a suitable element of $H$ and repeating the substitutions we used in Theorem \ref{thm: canform}, we may assume that $(b_2,\sigma)$ fixes $[1:0:0]$. This implies that $\gamma = 0$, hence $(1 + \delta)a = 0$ implies $a = 0$.
  Now, we plug everything into the equation for $\sigma^* a_6$ and compare coefficients to obtain the conditions $f = g = 0$.
 
 If $\delta = 1$, then $\gamma^4 + \gamma = 0$. Hence, if $(b_2,\sigma)$ is non-trivial, then $\gamma^3 = 1$. Modifying the equation of $X$ by an element of $H$, we may assume $\gamma = 1$, that is, that $(b_2,\sigma)$ maps $[1:0:0]$ to $[1:1:1]$. Then, Equation \eqref{eqn: 1a1} implies $b = 1$. Plugging into the equation for $\sigma^* a_6$ and comparing coefficients yields $g = f + a$ and $h = 1$. The square of both lifts of $(b_2,\sigma)$ to $\Aut(X)$ is the Bertini involution, hence the subgroup generated by these lifts is isomorphic to $C_4$.
 
 Suppose next that $G(X)$ contains two distinct non-trivial automorphisms with $\delta = 1$. Then, we can assume that one of them acts as in the previous paragraph, so $b = h = 1$ and $g = f + a$. The other one satisfies $\gamma \neq 1$. Plugging this into the equation for $\sigma^* a_6$ and comparing coefficents yields $f = a$. As in the previous paragraph, the square of all lifts of these automorphisms is the Bertini involution, hence they generate a subgroup isomorphic to the quaternion group $Q_8$.
 
 Finally, Corollary \ref{cor: preliminary} shows that $G(X)$ acts on the four singular points of $R$ through $A_4$, so if $G(X)$ contains a non-trivial automorphism with $\delta = 1$ and a non-trivial automorphism with $\delta \neq 1$, then $G(X) \cong A_4$. In particular, the previous two paragraphs show that $b = h = 1$ and $g = 0$ and $f = a$, while the above paragraph for $\delta \neq 1$ shows $a = f = g = 0$. In this case, $\Aut(X) \cong \SL_2(3)$.
 
\smallskip
\noindent \underline{Cases (1) (b) and (1) (c)}
\smallskip

\noindent In these cases, the singularity of $R$ over $[1:0:0]$ is not isomorphic to the other singularities of $R$, hence $G(X)$ is a subgroup of $C_3$ acting through the subgroup of $H$ with $\gamma = 0$. In particular, $G(X)$ fixes the points $[1:0:0]$ and $[0:1:0]$. Since the number of singular points of $R$ that lie over points different from $[1:0:0]$ and $[0:1:0]$ is not divisible by $3$, $G(X)$ fixes all of them, hence $G(X)$ is trivial.

\smallskip
\noindent \underline{Case (1) (d)}
\smallskip

\noindent In this case, $R$ has singularities over $[1:0:0]$ and $[1:c^{\frac{1}{4}}:c^{\frac{3}{4}}]$. An element of $H$ that fixes both of these points is trivial, and the unique one that swaps the two points is of the form $(b_2,\sigma)$ where $\sigma$ acts as $v \mapsto v + c^{\frac{1}{4}} u$ and $b_2 = c^{\frac{3}{4}} u^2 + c^{\frac{1}{2}}uv + c^{\frac{1}{4}} v^2$. If such an element lies in $G(X)$, then there exist polynomials $b_1$ and $b_3$ such that
\begin{eqnarray*}
0 &=& (a c^{\frac{1}{2}} + c^{\frac{3}{4}}) u^2 + ub_1 + b_1^2 + c^{\frac{1}{2}} uv + c^{\frac{1}{4}} v^2 \\
0 &=&  v^3b_1 + ub_3 + c u^2v^2 + c^{\frac{1}{2}} v^4, \\
\end{eqnarray*}
hence $b_1 = \lambda u + c^{\frac{1}{2}}v$ with $\lambda^2 + \lambda = ac^{\frac{1}{2}} + c^{\frac{3}{4}}$ and $c^4 = c$, and $b_3 = \lambda v^3 + cuv^2$. By Theorem \ref{thm: canform} we have $c \neq 0$, hence we can apply an element of $H$ to assume that $c = 1$.
 Plugging this into the equation for $\sigma^* a_6$ and comparing coefficients in Lemma \ref{lem: automorphisms}, we obtain the conditions $h = 0$ and $g = a + f$. The square of this $(b_1,b_2,b_3,\sigma)$ is the Bertini involution, hence $\Aut(X) \cong C_4$ in this case.

\smallskip
\noindent \underline{Case (1) (e)}
\smallskip

\noindent In this case, we have $G(X) \subseteq C_3$, since $G(X)$ fixes $[1:0:0]$.
    Non-trivial elements of $H$ that fix $[1:0:0]$ are of the form $(0,\sigma)$, where $\sigma$ acts as $v \mapsto \delta v$ with $\delta^3 = 1$ and $\delta \neq 1$. Such an automorphism lifts to $X$ if and only if there exist polynomials $b_1$ and $b_3$ such that
    \begin{eqnarray*}
    (1 + \delta^2)a v^2 &=& ub_1 + b_1^2 \\
    0 &=& v^3b_1 + ub_3 \\
    (1 + \delta^2)f u^4v^2 + (1 + \delta)gu^2v^4 &=& v^3b_3 + b_3^2.
    \end{eqnarray*}
    The first equation implies $a = 0$ and $b_1 = \lambda u$ with $\lambda^2 + \lambda = 0$ and then the second equation implies that also $b_3 = \lambda v^3$. Finally, the third equation shows $f = g = 0$.

\smallskip
\noindent \underline{Case (2) (a)}
\smallskip

\noindent
Here, $G(X) \subseteq H$ fixes the point $[0:1:1]$. Moreover, if $G(X)$ fixes the images of the other two singularities, then, by our description of $H$, $G(X)$ is trivial. Hence, $G(X) \subseteq C_2$ with equality if and only if $G(X)$ contains the involution $(0,\sigma)$, where $\sigma$ acts as $v \mapsto v + e^{\frac{1}{2}}g^{-\frac{1}{2}} u$.

If this involution is in $G(X)$, then there exist polynomials $b_1$ and $b_3$ such that
\begin{eqnarray*}
aeg^{-1} u^2 &=&  ub_1 + b_1^2 \\
e^2g^{-2} u^4 &=& ub_3,
\end{eqnarray*}
hence $b_1 = \lambda u$ with $\lambda^2 + \lambda = aeg^{-1}$, and $b_3 = e^2g^{-2} u^3$. Plugging this into the equation for $\sigma^* a_6$ and comparing coefficients in Lemma \ref{lem: automorphisms}, we obtain the conditions
 \begin{eqnarray*}
0 &=& e^4 + he^3g + f e^2g^2 + deg^3 \\
0 &=& e^{\frac{1}{2}}(g^{\frac{3}{2}} + he^{\frac{1}{2}}).
 \end{eqnarray*}
 Since $e \neq 0$ by Theorem \ref{thm: canform}, we have $h = e^{-\frac{1}{2}} g^\frac{3}{2}$ and $d = efg^{-1} + e^\frac{3}{2} g^{-\frac{1}{2}} +  e^3g^{-3}$. Note that both lifts of $(0,\sigma)$ have order $2$, hence $\Aut(X) \cong 2^2$.
 
\smallskip
\noindent \underline{Case (2) (b) and (2) (c)}
\smallskip

\noindent
Here, $G(X) \subseteq H$ fixes $[0:1:1]$ and $[1:0:0]$, since these are the points that lie under the singularities of the irreducible components of $R$, but not under the intersection of the two components $R_1$ and $R_2$. By our description of $H$ in Lemma \ref{lem: stabilizerofB}, this implies that $G(X)$ is trivial.

\smallskip
\noindent \underline{Case (2) (d)}
\smallskip

\noindent In this case, $G(X)$ fixes $[0:1:1]$, but we get no other restrictions from the position of the singularities of $R$. Therefore, an element of $G(X) \subseteq H$ is of the form $(0,\sigma)$ where $\sigma$ acts as $v \mapsto v + \gamma u$ for some $\gamma \in \Bbbk$. Such an element is in $G(X)$, if and only if there exist polynomials $b_1$ and $b_3$ such that
\begin{eqnarray*}
a \gamma^2 u^2 &=& ub_1 + b_1^2 \\
\gamma^4 u^4 &=& ub_3 \\
(c \gamma + d \gamma^2 + f \gamma^4 + h\gamma^6)u^6 + h\gamma^4 u^4v^2 + h \gamma^2 u^2v^4 &=& b_3^2.
\end{eqnarray*}
Such $b_1$ and $b_3$ exist if and only if $h = 0$ and $\gamma^8 + h \gamma^6 + f \gamma^4 + d \gamma^2 + c \gamma = 0$, and then $b_1 = \lambda u$ with $\lambda^2 + \lambda = a \gamma^2$, and $b_3 = \gamma^4 u^3$. By Theorem \ref{thm: canform}, we have $c \neq 0$, hence, as soon as $h = 0$, there are exactly $8$ choices for $\gamma$. This shows $G(X) \cong 2^3$. Every lift of every non-trivial element in $G(X)$ has order $2$, hence $\Aut(X) \cong 2^4$.
 
\smallskip
\noindent \underline{Case (2) (e)}
\smallskip

\noindent
Here, the elements of $G(X) \subseteq H$ fix $[0:1:0]$ and preserve the pair $\{[1:0:0],[1:1:0]\}$. Using our description of $\gamma$, it is clear that an element of $H$ that fixes all of these three points is the identity. An element that swaps $[1:0:0]$ and $[1:1:0]$ is of the form $(0,\sigma)$, where $\sigma$ acts as $v \mapsto  v + u$. 
Such an element is in $G(X)$, if and only if there exist polynomials $b_1$ and $b_3$ such that
\begin{eqnarray*}
a u^2 &=& ub_1 + b_1^2 \\
0 &=& ub_3 \\
(d + f + h)u^6 + (e + h)u^4v^2 + (e + h)u^2v^4 &=& b_3^2,
\end{eqnarray*}
hence if and only if $h = e$ and $d = e + f$, and then $b_1 = \lambda u$ with $\lambda^2 + \lambda = a$ and $b_3 = 0$. The square of the lift of this automorphism to $\Aut(X)$ is the identity, hence $\Aut(X) \cong 2^2$.

\smallskip
\noindent \underline{Case (2) (f)}
\smallskip

\noindent
In this case, $G(X) \subseteq H$ fixes $[1:0:0]$ and $[0:1:0]$. Hence, by our description of $H$ in Lemma \ref{lem: stabilizerofB}, every element in $G(X)$ is of the form $(0,\sigma)$, where $\sigma$ acts as $v \mapsto \delta v$ for some $\delta \in \Bbbk^{\times}$. A non-trivial element of this form is in $G(X)$ if and only if there exist $b_1$ and $b_3$ such that
\begin{eqnarray*}
a(1 + \delta^2)v^2 &=& ub_1 + b_1^2 \\
0 &=& ub_3 \\
d(1 + \delta^2)u^4v^2 + f(1 + \delta^4)u^2v^4 + (1 + \delta^5)uv^5 + h(1 + \delta^6) &=& b_3^2.
\end{eqnarray*}
Hence, we always have $b_1 = b_3 = 0$ and $\delta^5 = 1$. Since $\delta \neq 1$ by assumption, we deduce that $(0,\sigma)$ lifts if and only if $a = d = f = h = 0$.

\smallskip
\noindent \underline{Case (3)}
\smallskip

\noindent
Here, the group $G(X) \subseteq H$ fixes $[1:0:0], [0:1:0]$, and $[1:1:0]$. Hence, every element of $G(X)$ is of the form $(0,\sigma)$ and $\sigma$ satisfies the conditions of Lemma \ref{lem: stabilizerofB} (3).

First, assume that $\sigma$ has even order and interchanges two components of $B$. Without loss of generality, we may assume that $\sigma$ swaps $u$ and $v$. Then, $(0,\sigma)$ lifts to $X$ if and only if there exist $b_1$, and $b_3$ such that
\begin{eqnarray*}
0 &=& b_1^2  \\
(b+c)(u^3v+ uv^3) &=& uv(u+v)b_1 \\
(d+f)(u^5v+ uv^5) &=& uv(u+v)b_3 + b_3^2.
\end{eqnarray*}
This holds if and only if $b = c$, and then $b_1 = 0$ as well as $b_3 = \lambda uv(u+v)$ with $\lambda^2 + \lambda = 0$ and $d = f$. The square of both lifts of $(0,\sigma)$ is the identity, hence they generate a group isomorphic to $2^2$.

Next, assume that $\sigma$ is non-trivial and preserves the three components of $B$. Then, it acts as $u \mapsto \alpha u$, $v \mapsto \alpha v$, where $\alpha^3 = 1$, $\alpha \neq 1$. This automorphism lifts to $X$ if and only if there exist polynomials $b_1$ and $b_3$ such that
\begin{eqnarray*}
a(1 + \alpha^{-1})uv &=& b_1^2\\
(1+\alpha)(b u^3v+ (b+c) u^2v^2 + c uv^3) &=& uv(u+v)b_1 \\
0 &=& uv(u+v)b_3 + b_3^2,
\end{eqnarray*}
hence if and only if $a = b = c = 0$.

Finally, assume that $\sigma$ has odd order and interchanges components of $B$. Without loss of generality, we may assume that $\sigma$ acts as $u \mapsto \beta v, v \mapsto \beta(u + v)$ with $\beta^3 = 1$.
This lifts to $X$ if and only if there exist $b_1$ and $b_3$ such that
\begin{eqnarray*}
a ( 1 + \beta^2)uv + a\beta^2v^2 &=& b_1^2 \\
(b + \beta c)u^3v+ (b + c + \beta b)u^2v^2 + (c + \beta(b+c))uv^3 &=&  uv(u+v)b_1 \\
(d + f)u^5 v + f u^4v^2 + eu^2v^4 + (d + e)uv^5 + (d + e + f)v^6  &=& uv(u+v)b_3  + b_3^2.
\end{eqnarray*}
The third equation implies $f = d$ and $d = e + e^{\frac{1}{2}}$ and then $b_3 = \lambda u^2v+ \lambda uv^2 + e^{\frac{1}{2}} v^3$, where $\lambda^2 + \lambda = e + e^{\frac{1}{2}}$.
If $\beta = 1$, the first equation implies $b_1 = a^{\frac{1}{2}}v$ and the second equation implies $b = c = a^{\frac{1}{2}}$.
If $\beta \neq 1$, the first equation implies $b_1 = a = 0$ and the second equation implies $b = \beta c$.

\smallskip
\noindent \underline{Case (4)}
\smallskip

\noindent
In this case, the group $G(X) \subseteq H$ fixes $[1:0:0]$ and $[0:1:0]$, hence every element of $G(X)$ is of the form $(b_2, \sigma)$ with $b_2 = \lambda uv$ for some $\lambda \in \Bbbk$ and where $\sigma$ acts as $u \mapsto \alpha u, v \mapsto \alpha^{-2} v$ with $\alpha \in \Bbbk^{\times}$.

If such an automorphism lifts to $X$, then the condition $\sigma^*a_2 = a_2 + b_1^2 + b_2$ forces $b_2 = b_1^2$, hence $b_1 = b_2 = 0$. The other conditions of Lemma \ref{lem: automorphisms} become 
\begin{eqnarray*}
a(1 + \alpha)u^3v+ b(1 + \alpha^{-2})u^2v^2 + c(1 + \alpha^{-5})uv^3 &=& 0 \\
d(1 + \alpha^3)u^5v+ e(1 + \alpha^{-3}) u^3v^3 + f(1 + \alpha^{-9}) uv^5 &=& u^2vb_3 + b_3^2.
\end{eqnarray*}
Since $d \neq 0$, the second equation implies $\alpha^3 = 1$. Hence, if $\sigma$ is non-trivial, then $(0,\sigma)$ lifts to $X$ if and only if $a = b = c = 0$. 

\smallskip
\noindent \underline{Case (5)}
\smallskip

\noindent Here, $G(X) \subseteq H$ fixes $[0:1:0]$, hence every element of $G(X)$ is of the form $(b_2, \sigma)$ with $b_2 = \lambda u^2 + \mu uv$ for some $\lambda, \mu \in \Bbbk$ and where $\sigma$ acts as $u \mapsto \alpha u, v \mapsto \gamma u + \delta v$ with $\alpha^3 = 1, \gamma \in \Bbbk, \delta \in \Bbbk^{\times}$.

If such an automorphism lifts to $X$, then there exists $b_1$ with $b_1^2 + b_2 = 0$, hence $\mu = 0$ and $b_1 = \lambda^{\frac{1}{2}} u$. Comparing coefficients in the equation for $\sigma^* a_4$, we obtain
\begin{eqnarray}
\lambda^2 + \lambda^{\frac{1}{2}} + a \gamma + b \alpha^2 \gamma^2 + c \alpha \gamma^3 &=& 0 \label{eqn: 5a41} \\
a + a \delta + c \alpha \delta \gamma^2 &=& 0 \label{eqn: 5a42}  \\
b + b \alpha^2 \delta^2 + c \alpha \delta^2 \gamma &=& 0 \label{eqn: 5a43} \\
c + c \alpha \delta^3 &=& 0 \label{eqn: 5a44} 
\end{eqnarray}
The automorphism lifts to $X$ if and only if, additionally, there exists a $b_3 = \lambda_0 u^3 + \lambda_1 u^2v+ \lambda_2 uv^2 + \lambda_3 v^3$ satisfying the following conditions: 
\begin{eqnarray}
\lambda_0^2 + \lambda_0 &=& \lambda^3  + (a \gamma + b \alpha^2 \gamma^2 + c \alpha \gamma^3) \lambda + \alpha \gamma^5 + d\gamma^6 \label{eqn: 5a61}  \nonumber \\
\lambda_1 &=&  (a \delta + c \alpha \delta \gamma^2) \lambda + \alpha \delta \gamma^4 \label{eqn: 5a62} \nonumber \\
\lambda_1^2 + \lambda_2 &=& (b \alpha^2 \delta^2 + c \alpha \delta^2 \gamma) \lambda + d \delta^2 \gamma^4 \label{eqn: 5a63} \nonumber \\
\lambda_3 &=& c \alpha \delta^3 \lambda  \label{eqn: 5a64} \\
\lambda_2^2 &=& \alpha \delta^4 \gamma + d \delta^4 \gamma^2 \label{eqn: 5a65}  \nonumber \\
0 &=& 1 + \alpha \delta^5 \label{eqn: 5a66}  \\
\lambda_3^2 &=& d + d \delta^6 \label{eqn: 5a67}
\end{eqnarray}
Equation \eqref{eqn: 5a66} shows that $\alpha = \delta^{-5}$. In particular, as 
$\alpha^3 = 1$, we have $\delta^{15} = 1$.

First, assume that $\delta = 1$, hence $\alpha = 1$. Then, Equation \eqref{eqn: 5a67} shows that 
$\lambda_3 = 0$. Equation \eqref{eqn: 5a64} shows $c\lambda = 0$ and Equation \eqref{eqn: 5a42} shows $c \gamma = 0$. 
Hence, if $c \neq 0$, then $(b_2,\sigma)$ is the identity, so we assume $c = 0$ in the following. 
Let $G_{a,b,d}$ be the group of lifts of such automorphisms to $X$. By the description above, these $G_{a,b,d}$ 
form a family $\mathcal{G}_{a,b,d}$ of finite group schemes over ${\rm Spec} ~\Bbbk[a,b,d]$ cut out 
in ${\rm Spec} ~\Bbbk[a,b,d,\lambda,\lambda_0,\gamma]$ by the equations
\begin{eqnarray}
F_1 &:=&  \lambda^4 + \lambda + a^2 \gamma^2 + b^2 \gamma^4 =0  \label{eqn: 5A} \\
F_2 &:=& a^4 \lambda^4 + b^2 \lambda^2 + \gamma + d \gamma^2  + d^2 \gamma^8 + \gamma^{16} 
=0 \label{eqn: 5B} \\
F_3 &:=& \lambda_0^2 + \lambda_0 +  \lambda^3 + (a \gamma + b\gamma^2)\lambda + \gamma^5 + 
d \gamma^6 =0 \nonumber.
\end{eqnarray}

In the following, we show that all geometric fibers of $\mathcal{G}_{a,b,d} \to {\rm Spec} ~\Bbbk[a,b,d]$ are reduced of length $128$. In particular, $\mathcal{G}_{a,b,d}$ is \'etale over ${\rm Spec} ~\Bbbk[a,b,d]$, hence all the $G_{a,b,d}$ are isomorphic and we will show afterwards that $G_{a,b,d} \cong 2_+^{1+6}$.
 
\begin{itemize}
    \item If $a \neq 0$ and $b \neq a^2$, we argue as follows: The condition $a^8F_1^2 + F_2^2 + \frac{b^4}{a^4}F_2 = 0$ yields the following expression for $\lambda$:
\begin{eqnarray*}
(a^{12} + b^6) \lambda^2 &=& b^4 \gamma + (a^4 + b^4d) \gamma^2  + (a^4d^2 + b^4d^2 + a^{16}) \gamma^4  \\
&& + (a^{12}b^4 + b^4d^2)\gamma^8 + (a^4d^4 + b^4) \gamma^{16} + a^4 \gamma^{32}
\end{eqnarray*}
By our assumptions, we can divide by $(a^{12} + b^6)$ and we obtain an expression of $\lambda^2$ in terms of $\gamma$. Plugging this back into Equation \eqref{eqn: 5B}, we obtain a polynomial $F$ in $\gamma$ of the form $F = \sum_{i=0}^5 c_i \gamma^{2^i}$ of degree $64$ with
\begin{eqnarray*}
c_0 &=& 0 \\
c_1 &=& \frac{a^{12}}{a^{12} + b^6} \\ 
c_5 &=& \frac{a^8}{(a^{12} + b^6)^2}.
\end{eqnarray*}
Since $a \neq 0$, both $c_1$ and $c_5$ are non-zero, so $\partial_{\gamma} F = 1$ and $F$ has only simple roots. Hence, there are exactly $64$ choices for $\gamma$ such that $(b_2,\sigma)$ lifts and $\lambda$ is uniquely determined by $\gamma$. In particular, $G_{a,b,d}$ has order $128$ and it acts on the base of the associated elliptic fibration through $2^6$. 

\item If $a \neq 0$ and $b = a^2$, we argue as follows: The condition $a^8F_1^2 + F_2^2 + a^4F_2 = 0$ becomes
\begin{eqnarray*}
0 &=& a^8 \gamma + (a^4 + a^8d) \gamma^2  + (a^4d^2 + a^8d^2 + a^{16}) \gamma^4  \\
&& + (a^{20} + a^8d^2)\gamma^8 + (a^4d^4 + a^8) \gamma^{16} + a^4 \gamma^{32} =: F.
\end{eqnarray*}
Note that, since $a \neq 0$, $F_1 = F_2 = 0$ holds if and only if $F_2 = F = 0$. There are $32$ choices for $\gamma$ with $F(\gamma) = 0$ and for each choice of $\gamma$, there are exactly $2$ choices for $\lambda$ such that $F_2(\gamma,\lambda) = 0$. As in the previous case, $G_{a,b,d}$ has order $128$, but in this case, it acts on the base of the associated elliptic fibration through $2^5$.

\item Next, assume that $a = 0$ and $b \neq 0$. We can immediately solve Equation \eqref{eqn: 5B} for $\lambda$ and obtain
$$
b^2 \lambda^2 = \gamma + d\gamma^2 + d^2\gamma^8 + \gamma^{16}.
$$
Plugging this into the square of Equation \eqref{eqn: 5A}, we obtain a polynomial $F$ in $\gamma$ of the form $F = \sum_{i=0}^5 c_i \gamma^{2^i}$ of degree $64$ with
\begin{eqnarray*}
c_0 &=& 0 \\
c_1 &=& b^{-2} \\
c_5 &=& b^{-8}.
\end{eqnarray*}
Hence, there are $64$ choices for $\gamma$ such that $(b_2,\sigma)$ lifts and $\lambda$ is uniquely determined by $\gamma$. Therefore, $G_{a,b,d}$ has order $128$ and acts on the base of the associated elliptic fibration through $2^6.$

\item Now, assume that $a = b = 0$. The equations simplify to
\begin{eqnarray*}
\lambda^4 + \lambda &=& 0 \\
\lambda_0^2 + \lambda_0 &=& \lambda^3 + \gamma^5 + d \gamma^6 \\
\lambda_1 &=& \gamma^4 \\
\lambda_2 &=& d \gamma^4 + \gamma^8 \\
\lambda_3 &=& 0 \\
\gamma + d \gamma^2 + d^2 \gamma^8 + \gamma^{16} &=& 0.
\end{eqnarray*}
Hence, there are $16$ choices for $\gamma$ and $4$ choices for $\lambda$. Hence, $G_{a,b,d}$ has order $128$ and it acts on the base of the associated elliptic fibration through $2^4$.
\end{itemize}

It remains to determine the group $G_{a,b,d}$. By the 
last bullet point, the subgroup of $G_{0,0,0}$ of automorphisms that act trivially on the base of the associated elliptic fibration has order $8$. Thus, by Corollary \ref{cor: preliminary}, it is isomorphic to $Q_8$. Hence, every $G_{a,b,d}$ contains a quaternion group $Q_8$ with $\beta \in Q_8$. On the other hand, in the cases where $a \neq 0, b \neq a^2$, we have seen that $G_{a,b,d}/\langle \beta \rangle \cong 2^6$. Hence, by Example \ref{ex: extraspecial}, we have $G_{a,b,d} \cong  2_+^{1+6}$.
   
Next, assume that $\delta \neq 1$, $\delta^3 = 1$. Then, Equations \eqref{eqn: 5a42}, \eqref{eqn: 5a43}, and \eqref{eqn: 5a44} show that $a = b = c = 0$. The remaining equations become
\begin{eqnarray*}
\lambda^4 + \lambda &=& 0 \\
\lambda_0^2 + \lambda_0 &=& \lambda^3 + \delta \gamma^5 + d \gamma^6 \\
\lambda_1 &=& \delta^2 \gamma^4 \\
\lambda_2 &=& d \delta^2 \gamma^4 + \delta \gamma^8 \\
\lambda_3 &=& 0 \\
\delta^2 \gamma^{16} + d^2 \delta \gamma^8 + d \delta \gamma^2 + \delta^2 \gamma &=& 0. 
\end{eqnarray*}
We see that if $\gamma = \lambda = 0$, then $(b_2,\sigma)$ admits a lift to $X$ as an automorphism $g$ of order $3$. For a fixed $\gamma$, there are at most $128$ possible choices of $(\gamma,\lambda)$. All of them are obtained by composing $g$ with an element of $G_{0,0,d}$, hence all choices are realized.

Finally, assume that $\delta \neq 1$, $\delta^5 = \alpha = 1$. As in the previous paragraph, 
we have $a = b = c = 0$. But in this case, Equation \eqref{eqn: 5a67} yields the condition $d = 0$. 

So, in summary, if $c \neq 0$, then $G(X)$ is trivial and if $c = 0$, then $\Aut(X)$ admits a unique $2$-Sylow subgroup isomorphic 
to $2_+^{1+6}$. If $a,b$, or $c$ is non-zero, this is the full automorphism group. If $a = b = c = 0$ and $d \neq 0$, then $\Aut(X)/2_+^{1+6} \cong C_3$ and if $a = b = c = d = 0$, then $\Aut(X)/2_+^{1+6} \cong C_{15}$.
\end{proof}

\begin{remark} 
The largest order of an automorphism group of a del Pezzo surface of degree $1$ over the complex numbers is equal to 
$144$ and the surface with such a group of automorphisms is unique \cite{CAG}. In our case, the maximal order is equal 
to $1920 = 2^7\cdot 15$ and the surface with such an automorphism group is also unique. We also see the occurrence of 
the group $G = 2^4$ in Case (5). It is obtained as the pre-image in $2_+^{1+6}$ of a maximal isotropic subspace of 
$\bbF_2^6$. Since del Pezzo surfaces of degree $1$ are superrigid and the corresponding $G$-surface 
is minimal, this group is not conjugate in the Cremona group of $\bbP^2$ to the isomorphic subgroup of the group of automorphisms 
of del Pezzo surfaces of degree $4$ or $2$ that appeared in \cite{DD} and \cite{DM}. 
\end{remark} 

\subsection{Conjugacy classes and comparison with the classification in characteristic $0$}
In this section, we determine the conjugacy classes in $W(\sfE_8)$ of the elements of the groups that occur in Theorem \ref{thm: main1} and, whenever possible, compare the surfaces in Theorem \ref{thm: main1} with their counterparts in characteristic $0$ (see \cite[Table 8.14]{CAG}).

For a del Pezzo surface $X$ of degree $1$, we denoted by $K_X$ and $P_X$ the kernel and image of the morphism $\Aut(X) \to \Aut(\mathbb{P}^1)$ induced by the action of $\Aut(X)$ on the base of the associated elliptic pencil. First, we note the following:
\begin{lemma} \label{lem: conjugacyK}
Let $g$ be a non-trivial element of $K_X$. Then, the conjugacy class of $g$ is either $8A_1,4A_2,2D_4(a_1)$, or $E_8(a_8)$.
\end{lemma}
\begin{proof}
Since $g$ acts trivially on the base of the pencil, it cannot preserve any $(-1)$-curve on $X$. Then, the lemma follows from the classification of conjugacy classes in $W(\sfE_8)$ (see e.g. \cite[Table 3]{DM}), by checking which of them fix no $(-1)$-class in $\sfE_8$.
\end{proof}

\begin{corollary} \label{cor: specializationmap}
Let $X$ be a del Pezzo surface of degree $1$ in characteristic $2$. Let $X'$ be a geometric generic fiber of a lift of $X$ to characteristic $0$ and let ${\rm sp}: \Aut(X') \to \Aut(X)$ be the specialization map. Then, the following hold:
\begin{enumerate}
    \item ${\rm sp}$ is injective.
    \item ${\rm sp}$ induces morphisms $K_{X'} \to K_X$ and $P_{X'} \to P_X$.
    \item The kernel $H$ of $P_{X'} \to P_X$ is an elementary $2$-group and if $g$ is an element of $\Aut(X')$ that maps to a non-trivial element of $H$, then the conjugacy class of $g$ is $2D_4(a_1)$.
\end{enumerate}
\end{corollary}
\begin{proof}
Claim (1) follows from $H^0(X,T_X) = 0$. The existence of the morphisms in Claim (2) is clear.

 For Claim (3), recall that ${\rm sp}$ preserves conjugacy classes, so, by Lemma \ref{lem: conjugacyK}, all non-trivial elements of $H$ are represented by elements $g$ of $\Aut(X')$ of conjugacy class $8A_1,4A_2,2D_4(a_1)$, or $E_8(a_8)$. If $g$ is of class $8A_1$, then it is the Bertini involution, hence $g \in K_{X'}$. If $g$ is of class $4A_2$, then it has negative trace on $\sfE_8$, so, by the Lefschetz fixed point formula, it must act trivially on the base of the elliptic pencil. Hence, $g \in K_{X'}$. If $g$ is of class $E_8(a_8)$, then, by what we just proved, $g^2$ and $g^3$ are in $K_{X'}$, hence $g \in K_{X'}$. Thus, $g$ must be of conjugacy class $2D_4(a_1)$. Then, $g^2$ is the Bertini involution, so $H$ is $2$-elementary. 
\end{proof}

By Theorem \ref{thm: main1}, $|\Aut(X)| \leq 36$ or $|\Aut(X)| \in \{128,384,1920\}$, so Types I and II \cite[Table 8.14]{CAG} do not have a reduction modulo $2$ which is a del Pezzo surface.

The surfaces of Type VI, VII, IX, XII, and XV from \cite[Table 8.14]{CAG} admit an automorphism of order $2n$ with $n > 1$ acting faithfully on $\mathbb{P}^1$, which is impossible in characteristic $2$, so by Corollary \ref{cor: specializationmap} they do not have good reduction mod $2$. 

The equation of the surfaces of Type (3) (v) in Theorem \ref{thm: main1} can be rewritten as
$$
y^2 + uv(u-v)y + x^3 + a(u^2 - uv + v^2)^3 + bu^2v^2(u-v)^2
$$
for certain $a,b \in \Bbbk$. This equation makes sense in characteristic $0$, and it is stable under $\mathfrak{S}_3$-action generated by $(u,v,x,y) \mapsto (v,u,x,-y)$ and $(u,v,x,y) \mapsto (u-v,-u,x,-y)$ as well as the $C_3$-action $(u,v,x,y) \mapsto (u,v,\zeta_3 x,y )$, where $\zeta_3$ is a primitive $3$-rd root of unity. Hence, the automorphism group of has order at least $36$, hence it is isomorphic to $6 \times \mathfrak{S}_3$. Thus, surfaces of Type (3) (v) are reductions mod $2$ of the surfaces of Type III from \cite[Table 8.14]{CAG}. In particular, we can read off the conjugacy classes from \cite[Table 8]{DM}.

The equation of Type (5) (iii) makes sense in characteristic $0$, where it is isomorphic to
$$
y^2 + x^3 + u(u^5 + v^5),
$$
which is the equation of Type IV in \cite[Table 8.14]{CAG}.

The equation of the surfaces of Type (3) (ii) in Theorem \ref{thm: main1} can be rewritten as
$$
y^2 + uv(u-v)y + x^3 + c(u^2 - uv + v^2)x^2 + a(u^2 - uv + v^2)^3 + bu^2v^2(u-v)^2
$$
for certain $a,b,c \in \Bbbk$. Similar to the case of Type (3) (v) above, these equations are stable under a $\mathfrak{S}_3$-action, both in characteristic $0$ and in characteristic $2$. 
In characteristic $0$, these equations can be simplified to the normal forms of Type X from \cite[Table 8.14]{CAG}.

The equation of the surfaces of Type (3) (iii) makes sense in characteristic $0$, where it defines a lift of $X$ together with the action of $\Aut(X)$. Both $X$ and the lift admit an automorphism of order $6$ that acts trivially on the base of the elliptic pencil. Hence, these surfaces are reductions mod $2$ of the surfaces of Type XI from \cite[Table 8.14]{CAG}.

The equations of the surfaces of Type (2) (f) (i) in Theorem \ref{thm: main1} define a $1$-dimensional family of surfaces in characteristic $0$ with an action of $C_{10}$. These lifts must be of Type XIII \cite[Table 8.14]{CAG}.

The equations of the surfaces of Type (4) (i) in Theorem \ref{thm: main1} define a $2$-dimensional family of surfaces in characteristic $0$ with an action of $C_6$ that is trivial on the base of the elliptic pencil. Hence, these lifts are of Type XVII \cite[Table 8.14]{CAG}.

Next, consider the equations
$$
y^2 + (aux + bu^3 + cv^3) y + x^3 + (du^4 + euv^3) x + fu^6 + gu^3v^3 + hv^6,
$$
where $a,b,c,d,e,f,g,h$ are parameters. In characteristic $0$, we can simplify this equation to the normal form of Type XVIII from \cite[Table 8.14]{CAG}. In characteristic $2$, these equations cover three of the families of Theorem \ref{thm: main1}: If $a,c \neq 0$, we can simplify the equation to the normal form for Type (1) (a) (ii) which, in turn, specializes to Type (1) (e) (i) for special values of the parameters. If $a = 0$ but $b,c \neq 0$, we can simplify the equation to
$$
y^2 + (u^3 + v^3)y + x^3 + euv^3x + f u^6.
$$
This is an alternative normal form for our surfaces of Type (3) (iv).

Finally, consider the equations
\begin{eqnarray*}
&& y^2 + (a(u+v)x + b(u+v)^3 + c uv(u+v))y + x^3 + (d(u+v)^4 + euv(u+v)^2 + fu^2v^2)x \\
&+ &  (g(u+v)^6 + huv(u+v)^4 + iu^2v^2(u+v)^2 + ju^3v^3).
\end{eqnarray*}
In characteristic $0$, we can simplify this equation to the normal form of Type XX from \cite[Table 8.14]{CAG}. In characteristic $2$, these equations cover four of the families of Theorem \ref{thm: main1}: 

If $a \neq 0$, we can simplify the equation to
\begin{eqnarray*}
y^2 + (u+v)xy + x^3 + c uv x^2 +  du^2v^2 x + g(u+v)^6 + huv(u+v)^4 + iu^2v^2(u+v)^2 + ju^3v^3.
\end{eqnarray*}
If $d,j \neq 0$, we can rescale one of them to $1$ and obtain an alternative normal form for Type (2) (a) (i). If $d \neq 0$ and $j = 0$, we obtain a normal form for Type (2) (e) (i). If $j \neq 0$ and $d = 0$, we obtain an alternative normal form for Type (2) (d) (i). Note that $d = j = 0$ would lead to a singular surface. Since the family of Type (2) (d) (i) occurs as a reduction mod $2$ of certain surfaces of Type XX from \cite[Table 8.14]{CAG}, we call them Type XX'.

If $a = 0$ and $c \neq 0$, we can simplify the equation to
$$
y^2 + (b(u+v)^3 + uv(u+v))y + x^3 + (euv(u+v)^2 + fu^2v^2)x + (g(u+v)^6 + ju^3v^3).
$$
This defines a $4$-dimensional family of surfaces with $2^2$-action (one parameter is redundant). By Theorem \ref{thm: main1}, the corresponding surfaces must be of Type (3) (i).

The surfaces in the families $(1) (a) (i), (1) (a) (iii), (1) (a) (iv), (1) (d) (i), (5) (i), (5) (ii),$ and $(5) (iii)$ admit an automorphism of order $4$ and it turns out that writing down integral equations for such automorphisms similar to the ones above is hard. So, instead, to determine the conjugacy classes of the automorphisms of this family and to compare with the classification in characteristic $0$, we will use the following observation.

\begin{lemma} \label{lem: conjugacyclasses}
Let $g$ be an automorphism of a del Pezzo surface of degree $1$. Let $m \coloneqq {\rm ord}(g)$ and let $n$ be the order of the induced automorphism of $\mathbb{P}^1$. Assume that $m$ is even. Then, the conjugacy class $\Gamma$ of $g$ in $W(\sfE_8)$ is one of the following:
\begin{enumerate}
    \item If $(m,n) = (2,1)$, then $\Gamma = 8A_1$.
    \item If $(m,n) = (2,2)$, then $\Gamma = 4A_1$.
    \item If $m = 4$, then $\Gamma = 2D_4(a_1)$.
    \item If $(m,n) = (6,1)$, then $\Gamma = E_8(a_8)$.
    \item If $(m,n) = (6,2)$, then $\Gamma = E_6(a_2) + A_2$.
    \item If $(m,n) = (6,3)$ and $g^2$ is of class $3A_2$, then $\Gamma = E_7(a_4) + A_1$.
    \item If $(m,n) = (6,3)$ and $g^2$ is of class $2A_2$, then $\Gamma = 2D_4$.
    \item If $m = 10$, then $\Gamma = E_8(a_6)$.
    \item If $m = 12$, then $\Gamma = E_8(a_3)$.
    \item If $m = 20$, then $\Gamma = E_8(a_2)$.
    \item If $m = 30$, then $\Gamma = E_8$. 
\end{enumerate}
\end{lemma}
\begin{proof}
By Theorem \ref{thm: main1}, we know that the only possible values for $m$ and $n$ are the ones in the statement.

In Case (1), $g$ is the Bertini involution, hence $\Gamma = 8A_1$. In Case (2), we may assume that $g$ acts as $u \leftrightarrow v$. Then, we proved in this section that $g$ lifts to characteristic $0$, so by \cite[Table 8]{DM}, $\Gamma = 4A_1$. In Case (3), $g^2$ is the Bertini involution, because $\PGL_2(\Bbbk)$ does not contain elements of order $4$, and it is known (see \cite[Table 3]{DM}) that the only conjugacy class of automorphisms of order $4$ whose square is the Bertini involution is $2D_4(a_1)$. In Cases (8) and (11), $g^{\frac{m}{2}}$ is the Bertini involution and $g^2$ lifts to characteristic $0$, hence $g$ lifts to characteristic $0$ and we can read off the conjugacy class $\Gamma$ from \cite[Table 8]{DM}. Then, we deduce Case (10) from Case (8). In Case (9), $g^2$ must be of type $E_8(a_8)$, since $\PGL_2(\Bbbk)$ does not contain any elements of order $4$ or $6$. Then, from \cite[Table 3]{DM}, we see that $\Gamma = E_8(a_3)$. Finally, Cases (4), (5), (6), and (7) follow from \cite[Table 3]{DM} by comparing the conjugacy classes of $g^2$ and $g^3$.
\end{proof}

Now, we can complete Table \ref{tbl:autodp1} by using the description of $\Aut(X)$ in Theorem \ref{thm: main1}. We observe that the conjugacy classes for Types $(1) (a) (i)$ and $(1) (d) (i)$ are the same as for Type XIX from \cite[Table 8.14]{CAG}, the conjugacy classes for Type $(1) (a) (iii)$ are the same as for Type XIV from \cite[Table 8.14]{CAG}, and the conjugacy classes for Type $(1) (a) (iv)$ are the same as for Type V from \cite[Table 8.14]{CAG}. The only groups in Theorem \ref{thm: main1} that contain $D_8$ are $2_+^{1+6},2_+^{1+6} : 3$, and $2_+^{1+6} : 15$, and the only group that contains an automorphism of order $20$ is $2_+^{1+6} : 15$. Hence, if the Types XVI, M, and VIII from \cite[Table 8.14]{CAG} and \cite[Table 8]{DM} have good reduction modulo $2$, then they must reduce to our Types $(5)(i)$ and $(5)(ii)$, respectively. In each of these cases, we determine the conjugacy classes using Lemma \ref{lem: conjugacyclasses}.

We summarize the classification of automorphism groups of del Pezzo surfaces of degree $1$ in Table \ref{tbl:autodp2} in the Appendix. There, in the first column, we give the name of the corresponding family, both in the notation of Theorem \ref{thm: main2} and in the notation of \cite[Table 8.14]{CAG}. The second and third columns give the group $\Aut(X)$ and its size. In the remaining columns, we list the number of elements of a given Carter conjugacy class in $\Aut(X)$.

\newpage 
\begin{landscape}
\section*{Appendix}
\label{sec:appendix}

\vskip5pt
\begin{table}[htbp]
\centering
\renewcommand{\arraystretch}{1.3}
\scalebox{0.8}{%
\begin{tabular}{|c|cc|ccccccccccc|}
\hline
Name &$\Aut(X)$ & Order & ${\rm id}$ &  $2A_1$ & $4A_1$ & $A_2$& $A_2 + 2A_1$ &$A_3$  & $A_3 + A_1$ & $A_4$ & $D_4$ & $D_4(a_1)$ &$D_5$ \\
\hline \hline
$(\phi,\phi)$ & $2^4 : \frakA_5$ & $960$ & $1$  & $70$   & $5$ & $80$ & $80$ & & $120$ & $384$ & $160$ & $60$ & \\ \hline 
$(\zeta_3,\zeta_3)$ & \multicolumn{2}{c|}{\cellcolor{gray!20}(same as $(\phi,\phi)$)} & \multicolumn{11}{c|}{\cellcolor{gray!20}} \\ \hline
$(i,i)$ & \multicolumn{2}{c|}{\cellcolor{gray!20}(does not exist)} & \multicolumn{11}{c|}{\cellcolor{gray!20}} \\ \hline
$(a,a)$  & $2^4 : 2^2$ & $64$ &  $1$   & $22$ & $5$ & & & & $24$  & & & $12$&  \\ \hline
general  & $2^4$ & $16$ & $1$  & $10$   & $5$ &  & & & & & & & \\ \hline 
\end{tabular}
}
\caption{Automorphism groups of quartic del Pezzo surfaces}
\label{tbl:autodp4}
\end{table}

\begin{table}[htbp]
\centering
\renewcommand{\arraystretch}{1.3}
\scalebox{0.8}{%
\begin{tabular}{|c|cc|cccccccccccccccc|}
\hline
Name & $\Aut(X)$ & Order & ${\rm id}$  & $2A_1$ & $4A_1$ & $A_2$
& $A_2 + 2A_1$ &$2A_2$ & $3A_2$ & $A_3 + A_1$ & $A_4$ & $A_5 + A_1$ & $D_4$ & $D_4(a_1)$  & $D_5$ & $E_6 $ & $E_6(a_1)$ & $E_6(a_2)$ \\ \hline
\hline
I / 3C & $\PSU_4(2)$ & $25920$ & $1$ & $270$ &$45$  & $240 $ & $2160 $& $480$ & $80$ & $3240 $& $5184$ & $1440$ & $1440 $& $540$ &  & $4320$  &$5760 $ & $720$ \\
\hline
II / 5A &\multicolumn{2}{c|}{\cellcolor{gray!20}(same as V)} & \multicolumn{16}{c|}{\cellcolor{gray!20}} \\ \hline
III / 12A &  \multicolumn{2}{c|}{\cellcolor{gray!20}(same as I)} & \multicolumn{16}{c|}{\cellcolor{gray!20}} \\ \hline
IV /  3A &  $\calH_3(3) : 2$ & $54$ & $1$ &  & $9$  &   & & $24$ & $2$&  &  &  &  & &  &     &  & $18$ \\
\hline
V / 4B &$2^3 : \frak{S}_4$ & $192$ & $1$ & $30$  & $13$ &  &  & $32$ & & $72$ &  & $32$ &  & $12$ &  &     &  & \\ \hline
VI / 6E & \multicolumn{2}{c|}{\cellcolor{gray!20}(same as V)} & \multicolumn{16}{c|}{\cellcolor{gray!20}} \\ \hline
VII / 8A & \multicolumn{2}{c|}{\cellcolor{gray!20}(does not exist)} & \multicolumn{16}{c|}{\cellcolor{gray!20}} \\ \hline
 VIII / 3D &  $\frakS_3$ & $6$ & $1$ &  & $3$ &  & & $2$ &  &  &  &  &  &  &  &     &  & \\
\hline
IX / 4A & \multicolumn{2}{c|}{\cellcolor{gray!20}(same as V)} & \multicolumn{16}{c|}{\cellcolor{gray!20}} \\ \hline
X / 2B & $2^4$ & $16$ &$ 1$ & $10 $& $5 $&  &  &  &  &  &  &  &  &  &  &  &  & \\
\hline
XI / 2A &$2$ &$ 2 $& $1$ &  & $1$ &  &  &  &  &  &  &  &  &  &  &  &  & \\
\hline
XII / 1A &  $1$ &$ 1 $&$ 1$ &  &  &  &  &  &  &  &  &  &  &  &  &  &  & \\
\hline
\end{tabular}
}
\caption{Automorphism groups of cubic del Pezzo surfaces}
\label{tbl:cubics}
\end{table}
\end{landscape}

\newgeometry{left=3cm,right=3.4cm,bottom=0.5cm,top=1.6cm}

\begin{landscape}
\begin{table}[htbp]
\centering
\renewcommand{\arraystretch}{1.3}
\scalebox{0.7}{%
\begin{tabular}{|l|cc|cccccccccccccccccccccc|}
\hline
Name &  $\Aut(X)$ & Order & ${\rm id}$  & $3A_1$ & $4A_1$ & $7A_1$ & $2A_2$ & $3A_2$ &  $2A_3$ & $2A_3 + A_1$ & $A_5 + A_2$ & $A_6$ & $D_4(a_1)$ & $D_4(a_1) + A_1$ & $D_5$ & $D_5 + A_1$ & $D_6(a_2) + A_1$ & $E_6$ &$E_6(a_1)$ &$E_6(a_2)$ & $E_7$ & $E_7(a_1)$ & $E_7(a_2)$ &$E_7(a_4)$ \\
\hline \hline
I - V &  \multicolumn{2}{c|}{\cellcolor{gray!20}(do not exist)} & \multicolumn{22}{c|}{\cellcolor{gray!20}} \\ \hline
VI / $(3)(ii)$ &  $18$ & $18$ & $1$ & & & $1$ & & $2 $& & & & & & && &&& $6$ & & $6$&  &&  $2$ \\ \hline 
VII &  \multicolumn{2}{c|}{\cellcolor{gray!20}(does not exist)} & \multicolumn{22}{c|}{\cellcolor{gray!20}} \\ \hline
VIII / $(2)(a)(iii)$& $2 \times 6$ &$ 12 $& $1$ & $1 $&$ 1 $& $1 $& & $2 $& & & $2$ & && & & && && $ 2$ && &&  $2$ \\ \hline 
IX / $(1)(a)(ii)$ &  $2 \times \frakS_3$ & $12$ & $1$ & $3$ & $3$ & $1 $& $2$ & & & & & &&& & & $2$ && && & &&\\ \hline 
X / $(1)(c)(i)$ & $2^4$ & $16$&$ 1$ &$ 7$ &$ 7$ &$ 1$ & & & & & & &&&&& & & & && &&\\ \hline 
XI / $(2)(a)(ii), (3)(i)$ &$6$ & $6$ & $1$ & & & $1$ & & $2 $ & &&&  &&& & && & & & & &&$ 2 $\\ \hline 
XII / $(1)(a)(i), (2)(a)(i)$&  $2^2$ & $4$ & $1 $&$ 1$ &$ 1 $&$ 1$ & & & &&&&& & & && & & & && &\\ \hline 
XIII & $2$ &$ 2$ &$ 1 $& & & $1 $& &  &  && & &  && & & & && & && &\\ \hline 
\end{tabular}
}
\caption{Automorphism groups of del Pezzo surfaces of degree $2$}
\label{tbl:autodp2}
\end{table}

\begin{table}[htbp]
\centering
\renewcommand{\arraystretch}{1.3}
\scalebox{0.53}{%
\begin{tabular}{|l|cc|cccccccccccccccccccccccccc|}
\hline
Name  & $\Aut(X)$ & Order & ${\rm id}$  & $4A_1$ & $8A_1$ & $2A_2$  & $3A_2$& $4A_2$ & $2A_3 + A_1$ & $2A_4$ & $A_5 + A_2 + A_1$ & $D_4 + A_2$& $2D_4$ &$D_4(a_1) + A_1$  & $2D_4(a_1)$ & $D_8(a_3)$ & $E_6 + A_1$ & $E_6(a_2)$& $E_6(a_2) + A_2$ & $E_7(a_2)$ &$E_7(a_4) + A_1$ & $E_8$ & $E_8(a_1)$ &$E_8(a_2)$  & $E_8(a_3)$ & $E_8(a_5)$ &$E_8(a_6)$ &$E_8(a_8)$ \\
\hline \hline
I - II& \multicolumn{2}{c|}{\cellcolor{gray!20}(do not exist)} & \multicolumn{26}{c|}{\cellcolor{gray!20}} \\ \hline 

III / $(3) (v)$ &  $6 \times D_6$& $36 $&$ 1$ & $6 $&$ 1$ & $2$ & $4$ & $2$ & & & & & $2$ &  & & & & & $12 $& & $4$  & & & & & & & $2$ \\ \hline 

IV / $(5) (iii)$ &  $2_+^{1+6} : 15$ & $1920$ &$ 1$ & $70$ & $1$ & &  & $8$ & & $64$ & & & & & $56$ &  & & & $80$ & & & $512$ & & $384$ & $160$ & $512$ & $64$ & $8$\\   \hline 

M / $(5)(ii)$  & $2_+^{1+6} : 3$ & $384$ &$ 1$ & $70$ & $1$ & &  & $8$ & & & & & & & $56$ &  & & & $80$ & & & & & & $160$ & & & $8$\\   \hline 

V / $(1)(a)(iv)$ & $\SL_2(3)$& $24$ & $1 $& & $1 $& &$ 8$ & & & & & & & & $6$ & & & & & & $8$ & & & & & & & \\ \hline 

VI - VII & \multicolumn{2}{c|}{\cellcolor{gray!20}(do not exist)} & \multicolumn{26}{c|}{\cellcolor{gray!20}} \\ \hline 

VIII & \multicolumn{2}{c|}{\cellcolor{gray!20}(same as IV)} & \multicolumn{26}{c|}{\cellcolor{gray!20}} \\ \hline

IX & \multicolumn{2}{c|}{\cellcolor{gray!20}(does not exist)} & \multicolumn{26}{c|}{\cellcolor{gray!20}} \\ \hline 

X / $(3) (ii)$ &  $D_{12}$&$ 12 $& $1$ &$ 6$ & $1$ & $2$ & & & & & & & $2$ &  & & & & & & &  & & & & & & & \\ \hline 

XI / $(3) (iii)$ &$2 \times 6$& $12 $& $1 $& $2$ & $1 $& & &$ 2$ & & & & & & & & & & & $4$ & & & & & & & & & $2 $\\ \hline 

XII & \multicolumn{2}{c|}{\cellcolor{gray!20}(does not exist)} & \multicolumn{26}{c|}{\cellcolor{gray!20}} \\ \hline 

XIII / $(2) (f) (i)$  & $10$&$ 10$ & $1$ & & $1$ & & & & & $4$& & & & & & & & & & & & & & & & & $4$ & \\ \hline 

XIV / $(1) (a) (iii)$& $Q_8$& $8$ & $1$ & & $1$ & & & & & & & & & & $6$ & & & & & & & & & & & & & \\ \hline 
 
XV & \multicolumn{2}{c|}{\cellcolor{gray!20}(does not exist)} & \multicolumn{26}{c|}{\cellcolor{gray!20}} \\ \hline 

XVI / $(5)(i)$ & $2_+^{1+6}$&$ 128 $& $1$& $70$ & $1 $& & & & & & & & & & $56$& & & & & & & & & &  & & & \\ \hline 

XVII / $(4)(i)$ & $6$& $6$ &$ 1$ &  & $1$ & & & $2$ & & & & & & & & & & & & & & & & & & & & $2$ \\ \hline 

XVIII / $(1)(a)(ii),(1)(e)(i),(3)(iv)$ &  $6$&$ 6 $& $1 $&  & $1 $& & $2$ & & & & & & & & & & & & & & $2 $& & & & & & & \\ \hline 

XIX / $(1)(a)(i),(1)(d)(i)$&  $4$& $4 $& $1$& & $1$ & & & & & & & &  & & $2$ & & & & & & & & & & & && \\ \hline 

XX / $(2)(a)(i),(2)(e)(i),(3)(i)$ &  $2^2$ & $4$ & $1$ & $2$ & $1$ & & & & & & & & & & & & & & & & & & & & & & &  \\ \hline 

XX' / $(2)(d)(i)$ &  $2^4$ & $16$ & $1$ & $14$ & $1$ & & & & & & & & & & & & & & & & & & & & & & &  \\ \hline 

XXI &  $2$ &$ 2 $&$ 1 $& &$1 $& & &  & & & & & & & & & & & & & & & & & & & & \\ \hline

\end{tabular}
}
\caption{Automorphism groups of del Pezzo surfaces of degree $1$}
\label{tbl:autodp1}
\end{table}
\end{landscape}

\restoregeometry

\end{document}